\documentclass[a4paper,12pt,intlimits,oneside]{amsart}
\usepackage{amsmath}
\usepackage{amsthm}
\usepackage{latexsym}
\usepackage{amssymb}
\usepackage{xcolor}
\usepackage{graphicx}
\usepackage[colorlinks=true]{hyperref}
\numberwithin{figure}{section}
\def\R{{\mathbb R}}
\def\C{{\mathbb C}}

\def\T{{\mathbb T}}
\def\Z{{\mathbb Z}}

\def\N{{\mathbb N}}
\def\h{\mathfrak h}
\def\la{\langle}
\def\ra{\rangle}
\def\s{\vskip 0.25cm\noindent}
\def\e{\varepsilon}

\def\build#1_#2^#3{\mathrel{
\mathop{\kern 0pt#1}\limits_{#2}^{#3}}}
\def\td_#1,#2{\mathrel{\mathop{\build\longrightarrow_{#1\rightarrow #2}^{}}}}
\DeclareFontFamily{U}{MnSymbolC}{}
\DeclareSymbolFont{MnSyC}{U}{MnSymbolC}{m}{n}
\DeclareFontShape{U}{MnSymbolC}{m}{n}{
    <-6>  MnSymbolC5
   <6-7>  MnSymbolC6
   <7-8>  MnSymbolC7
   <8-9>  MnSymbolC8
   <9-10> MnSymbolC9
  <10-12> MnSymbolC10
  <12->   MnSymbolC12}{}
\DeclareMathSymbol{\intprod}{\mathbin}{MnSyC}{'270}
\newtheorem{theorem}{Theorem}
\newtheorem{corollary}{Corollary}
\newtheorem{proposition}{Proposition}
\newtheorem{lemma}{Lemma}
\newtheorem{remark}{Remark}

\begin{document}
\title[Smoothing properties and Tao's gauge transform]{On smoothing properties and Tao's gauge transform of 
the Benjamin-Ono equation on the torus}
\author[P. G\'erard]{Patrick G\'erard}
\address{Laboratoire de Math\'ematiques d'Orsay, CNRS, Universit\'e Paris--Saclay, 91405 Orsay, France} \email{{\tt patrick.gerard@universite-paris-saclay.fr}}
\author[T. Kappeler]{Thomas Kappeler}
\address{Institut f\"ur Mathematik, Universit\"at Z\"urich, Winterthurerstrasse 190, 8057 Zurich, Switzerland} 
\email{{\tt thomas.kappeler@math.uzh.ch}}
\author[P. Topalov]{Petar Topalov}
\address{Department of Mathematics, Northeastern University,
567 LA (Lake Hall), Boston, MA 0215, USA}
\email{{\tt p.topalov@northeastern.edu}}



\begin{abstract}
We prove smoothing properties of the solutions of the Benjamin-Ono equation
in the Sobolev space $H^{s}(\T, \R)$ for any $s \ge 0$. 
To this end we show
that Tao's gauge transform is a high frequency approximation 
of the nonlinear Fourier transform $\Phi$ for the Benjamin-Ono equation,
constructed in our previous work. 
The results of this paper are manifestations of the quasi-linear character of the Benjamin-Ono equation.
\end{abstract}


\thanks{
T.K.  partially supported by the Swiss National Science Foundation.
P.T.  partially supported by the Simons Foundation, Award \#526907.}

\maketitle

\tableofcontents

\medskip

\section{Introduction}\label{introduction}
In this paper we consider the Benjamin-Ono (BO) equation on the torus,
\begin{equation}\label{BO}
\partial_t v = \textup{H}[\partial^2_x v] - \partial_x v^2\,, \qquad x \in \T:= \R/2\pi\Z\,, \  t \in \R,
\end{equation}
where $v\equiv v(t, x)$ is real valued and $\textup{H}$ denotes the Hilbert transform, defined for $f = \sum_{n \in \mathbb Z} \widehat f(n) e^{ i n x} $,
$\widehat f(n) = \frac{1}{2\pi}\int_0^{2\pi} f(x) e^{-  i n x} dx$,  by
$$
\textup{H}[ f](x) := \sum_{n \in \Z} -i \ \text{sign}(n) \widehat f(n) \ e^{inx}
$$
with $\text{sign}(\pm n):= \pm 1$ for any $n \ge 1$, whereas $\text{sign}(0) := 0$. 

Equation \eqref{BO} 
has been introduced by Benjamin \cite{Benj} and Davis\&Acrivos \cite{DA} to model long, uni-directional internal waves 
in a two-layer fluid. It has been extensively studied, both on the real line $\R$ and on the torus $\T$. 
Let us briefly summarize some of the by now classical results 
on the well-posedness problem of \eqref{BO}, relevant for this paper -- we refer to \cite{Sa} for an excellent survey as well as a derivation of  \eqref{BO}.
Based on work of Saut \cite{Sa0},  Abdelouhab\&Bona\&Felland\&Saut proved in \cite{ABFS} that
for  $s > 3/2$, equation \eqref{BO} is globally in time well-posed on the Sobolev space $H^s_r \equiv H^s(\T, \R)$
(endowed with the standard norm $\| \cdot \|_s$, defined by \eqref{Hs norm} below), meaning the following:
\begin{itemize}
\item[(S1)]
{\em Existence and uniqueness of classical solutions:} For any initial data $v_0 \in H^{s}_r$, there exists 
a unique curve $v : \R \to H^s_r$ in
$C(\R, H^s_r) \cap C^1(\R, H^{s-2}_r)$ so that
$v(0) = v_0$ and for any $t \in \R$, equation \eqref{BO} is satisfied in $H^{s-2}_r$.   
(Since $H^s_r$ is an algebra, one has 
$\partial_x(v(t)^2 )\in H^{s-1}_r$ for any time $t \in \R$.)
\item[(S2)]{\em Continuity of solution map:}
The solution map 
$\mathcal S : H^s_r \to C(\mathbb R, H^s_r)$
is continuous, meaning that for any $v_0 \in H^s_r,$
$T > 0$, and $\varepsilon > 0$ there exists $\delta > 0,$ 
so that for any $w_0 \in H^s_r$ with $\|w_0 - v_0 \|_s < \delta$,
the solutions $w(t) = \mathcal S(t, w_0)$
 and 
$v(t)= \mathcal S(t, v_0)$ of \eqref{BO}
with initial data $w(0) = w_0$ and, respectively, $v(0) = v_0$
satisfy $\sup_{|t| \le T} \| w(t) - v(t) \|_s \le  \varepsilon$.
\end{itemize}
In the sequel, further progress has been made 
on the well-posedness of \eqref{BO} 
on Sobolev spaces of low regularity. 
The best results so far in this direction
were obtained by Molinet, using as a key ingredient the gauge transform, introduced by Tao \cite{Tao}
for the Benjamin-Ono equation on $\R$.
Molinet's results in \cite{Mol} (cf. also \cite{MP}) imply that the solution 
map $\mathcal S,$ introduced in $(S2)$ above, 
continuously extends to any
Sobolev space $H^s_r$ with $0 \le s \le 3/2$. 
More precisely, for any such $s$,
$\mathcal S: H^s_r \to C(\R, H^s_r)$ is continuous
and for any $v_0 \in H^s_r$, $ \mathcal S(t, v_0)$ satisfies equation \eqref{BO}  
in $H^{s-2}_r$. 
Finally, in the recent paper \cite{GKT1} we proved that \eqref{BO} is wellposed in the Sobolev space $H^s_r$ 
for any $s > -1/2$, but illposed for $s \le -1/2$.

In a straightforward way one verifies that for any solution $v(t) \equiv \mathcal S(t, v_0)$ of \eqref{BO} in $H^s_r$
with $s > -1/2$,  the mean $ \langle v(t) | 1 \rangle$
is conserved. Here $\langle \cdot  \, |  \, \cdot \rangle $ denotes 
the extension of the $L^2-$inner product, 
\begin{equation}\label{L2 inner product}
\langle f | g \rangle =  \frac{1}{2\pi} \int_0^{2\pi} f \overline g dx \, , \qquad \forall \, f, g \in L^2_c \equiv L^2(\T, \C)
\end{equation}
to the dual pairing $H^s_r \times H^{-s}_r \to \C$.
As a consequence,  for any $s > -1/2$, the subspace
\begin{equation}\label{affine spaces}
H^s_{r,0} :=  \{ v \in H^s_r \, : \, \langle v | 1 \rangle =0 \} \, , 
\end{equation}
of $H^s_r$ is invariant by the flow of \eqref{BO}. (For $s=0$, we usually write  $L^2_{r,0}$ for $H^0_{r,0}$.)
Since for any $a \in \R$ and any solution $v(t) = \mathcal S(t, v_0)$ of \eqref{BO} in $H^s_r$ with $s > -1/2$,
$v_a (t,x):=a+v(t, x - 2at)$ is again a solution of \eqref{BO} in $H^s_r$, for our purposes, 
it suffices to consider solutions in $H^s_{r,0}$.

The main goal of this paper is to prove smoothing properties of solutions of \eqref{BO}.  
A first key ingredient in their proof is Tao's gauge transform, which we denote by  $\mathcal G$.
To define it, 
we first need to introduce some more notation. For any $f \in H^s_{c} \equiv H^s(\T, \C)$, $s \in \R$,
the Szeg\H{o} projection $\Pi f$ of $f = \sum_{n \in \Z} \widehat f(n) e^{inx}$ is defined as $\sum_{n \ge 0} \widehat f(n) e^{inx}$.
Clearly, $\Pi$ defines a bounded linear operator $H^s_c \to H^s_+$ where
$$
H^s_+ :=  \{ f \in H^s_c \ : \ \widehat f(n) = 0 \ \  \forall \, n < 0  \} \, .
$$
Furthermore, we denote 
by $\partial_x^{-1}$ the operator
$$
\partial_x^{-1} : H^s_{c} \to H^{s+1}_{c,0}, f \mapsto \sum_{n \ne 0} \frac{1}{in} \widehat f(n) e^{inx}\, ,
$$
where for any $s \in \R$, $H^{s}_{c,0} :=  \{ v \in H^s_c \, : \, \langle v | 1 \rangle =0 \}$.
For notational convenience, the restriction of $\partial_x^{-1}$ to $H^{s}_{c,0}$ is also denoted by $\partial_x^{-1}$.

For our purposes, it suffices to consider solutions of \eqref{BO} in the Sobolev spaces $H^s_{r, 0}$ with $s \ge 0$.
For any $u \in H^s_{r, 0}$ with $s \ge 0$, 
we denote by $\mathcal G(u)$ (the following version of)
Tao's gauge transform of $u$  (cf. \cite{Tao}, \cite{MP}), 
\begin{equation}\label{gauge}
\mathcal G(u) : = \partial_x \Pi e^{-i \partial_x^{-1}u}\, .
\end{equation}
It was pointed out in \cite{Tao} that $\mathcal G$ can be viewed as a complex version of the Cole-Hopf transform,
which was introduced independently by Cole and Hopf in the early fifties to convert Burgers' equation $\partial_tu = \partial_x(\partial_x u - u^2)$
into the heat equation -- see e.g. \cite[Section 4.4]{Evans}.
\s
Note that
$$
\partial_x \Pi [e^{-i \partial_x^{-1}u}] = \Pi [\partial_x e^{-i \partial_x^{-1}u}] = -i \Pi [u e^{- i \partial_x^{-1}u}] 
$$
and that for any $s \ge 0$,
$$
\mathcal G : H^s_{r,0} \to H^s_{+,0}\, , \, u \mapsto \partial_x \Pi e^{-i \partial_x^{-1}u} \, ,  
$$
is a real analytic map, where 
$$
H^s_{+,0}:= \{ f \in H^s_{+,0} \, : \, \la f | 1 \ra = 0 \} \, , \qquad H_{+,0} \equiv H^0_{+,0} \, .
$$
It turns out that for any $s \ge 0$, $\mathcal G$ is a diffeomorphism onto an open proper subset of $H^s_{+,0}$ --
see Appendix \ref{Appendix C} for a proof.

Given any initial data $u_0 \in H^s_{r, 0}$ with $s \ge 0$, let $u(t) = \mathcal S(t, u_0)$
and denote by $ w(t) = \mathcal G(u(t))$
the gauge transform of $ u(t)$, i.e.,
$$
w(t) = \partial_x \Pi [e^{-i \partial_x^{-1}u(t)} ]  \, .
$$
For notational convenience, we will often not explicitly indicate the dependence of $u$, $v$, and $w$ on $t$
 in the sequel. Let us derive the equation, satisfied by $w(t)$.
Since $\partial_x^{-1} \partial_x (u^2) = u^2  - \langle u^2 | 1\rangle$ one sees that 
$v(t):= \partial_x^{-1}u(t)$ satisfies
\begin{equation}\label{equation v}
\partial_t v = \textup{H}[\partial^2_x v] - (\partial_x v)^2 + \langle (\partial_x v)^2  | 1 \rangle\, , \qquad
v(0) = \partial_x^{-1} u_0 \, .
\end{equation}
Furthermore, using that $\partial_t w = \partial_x \Pi [-i \partial_t v \cdot e^{-i v}]$ and
$$
\partial_x^2 w = \partial_x \Pi [-i \partial^2_x v \cdot e^{-i v} - (\partial_x v)^2 e^{-i v} ]\, ,
$$
one computes
$$
\partial_t w + i \partial^2_x w = \partial_x \Pi [-i \partial_t v \cdot e^{-i v} + \partial^2_x v \cdot e^{-i v} - i(\partial_x v)^2 e^{-i v} ] \, .
$$
Since for any $f \in H^s_c$, the Hilbert transform $\textup{H}[f]$ of $f$ satisfies $\textup{H}[f] = - i f + 2i(\text{Id} - \Pi) [f]$ one infers that
$$
\partial^2_x v =  i \textup{H}[ \partial^2_x v] + 2 (\text{Id} - \Pi)[ \partial^2_x v] \, .
$$
Combining the latter identity with \eqref{equation v} then yields
$$
\partial_t w + i \partial_x^2 w =  \partial_x \Pi \big[ -i  \langle (\partial_x v)^2  | 1 \rangle e^{-i v}  + 2 ( e^{-i v} \cdot (\text{Id} - \Pi) (\partial^2_x v)  \big] \, .
$$
Finally, writing 
$$e^{-i v} = \Pi e^{-i v} + (\text{Id} -\Pi) [e^{-i v}] \, , \quad 
\Pi e^{-i v} = \partial_x^{-1} w + \langle e^{-iv} | 1 \rangle \, , 
$$
and using that
$$
\Pi \big[ \langle e^{-iv} | 1 \rangle \cdot (\text{Id} - \Pi) (\partial^2_x v) ] = 0 \, , \quad
    \Pi \big[ (\text{Id} - \Pi) e^{-i v}  \cdot (\text{Id} - \Pi) (\partial^2_x v) ] = 0 \, , 
$$
one arrives at
$$
\partial_t w + i \partial_x^2 w =  - i  \langle (\partial_x v)^2  | 1 \rangle w 
+ 2 \partial_x \Pi  [  \partial_x^{-1} w \cdot (\text{Id} - \Pi) (\partial^2_x v) ]\, ,
$$
or, expressing the latter equation in terms of $u = \partial_x v$ instead of $v$, 
\begin{equation}\label{equation w}
\partial_t w + i \partial_x^2 w + i  \langle u^2  | 1 \rangle w  =
2 \partial_x \Pi [  \partial_x^{-1} w \cdot (\text{Id} - \Pi) (\partial_x u) ]\, .
\end{equation}
One verifies in a straightforward way  that $\langle u^2  | 1 \rangle$ is conserved along the flow of \eqref{BO}
so that the left hand side of \eqref{equation w} can be viewed as a linear expression in $w$ with constant coefficients.
We are now ready to state the smoothing properties of $u(t)$. 
\subsection{Approximation of $u(t)$}
For any $u_0 \in H^s_{r,0}$ with $s \ge 0$, let $w_0 := \mathcal G(u_0)$.
Furthermore, denote by $w_L(t)$ the solution of the linear initial value problem
$$
\partial_t w + i \partial_x^2 w + i  \langle u_0^2  | 1 \rangle \, w  = 0 \, , \qquad w(0) = w_0\, .
$$
Then $w_L(t)$ is given by
\begin{equation}\label{def w_L}
w_L(t) = \sum_{n \ge 1} e^{it(n^2 - \langle u_0^2 | 1 \rangle)} \widehat w_0(n) e^{inx} \, .
\end{equation}
Finally, we define 
\begin{equation}\label{sigma}
\sigma(s) := \begin{cases}
1 \qquad \quad \  {\text{if}} \ \  s > 1/2 \\
1- \qquad   {\text{if}} \ \ s = 1/2 \\
2s \qquad \ \ \,   {\text{if}} \ \ 0 \leq  s < 1/2
\end{cases} 
\end{equation}
where $a-$ means $a-\e$ for any $\e >0$. 
\begin{theorem}\label{smoothing 1}
For any $u_0 \in H^s_{r,0}$ with $s\geq 0$ there exists $M_s> 0$ so that 
for any $t \in \R$,
\begin{equation}\label{approx 1}
 \| w(t) -  w_L(t) \|_{s+\sigma(s)} \le  M_s  \la t \ra \, , \qquad  \la t \ra := 1 + |t| \, ,
\end{equation}
\begin{equation}\label{smoothing 1 of u by w_L B}
u(t) = 2 {\rm Re} \big(  e^{i \partial_x^{-1} u(t)} i w_L(t)  \big) + r(t) \, , \qquad \| r(t) \|_{s +\sigma(s)} \le  M_s  \langle t \rangle \, .
\end{equation}
The constant $M_s > 0$ can be chosen uniformly for bounded subsets of initial data $u_0$ in $H^s_{r,0}$. 
Furthermore, for any $0 \le s < 1/2$, there exists $u_0\in H^s_{r,0}$ so that for any $t\ne 0$ and any $\e >0$,
$w(t)-w_L(t)$ does not belong to 
$H_+^{s+\sigma (s)+\e }$ and $r(t)$ not  to $H_r^{s+\sigma (s)+\e }$. 
\end{theorem}
\begin{remark}
The estimate \eqref{approx 1} in Theorem \ref{smoothing 1} improves on \cite[Theorem 1.2]{IMOS} in the following ways:
(i) the estimate holds for $H^s_{r,0}$ with $s > 0$ arbitrary instead of $1/6 < s \le 1$;
(ii) the estimate holds for any $t \in \R$ with an explicit growth rate in $t$ instead for compact time intervals $[0, T]$;
(iii) for $1/2 < s \le 1$, the order of smoothing is $1$ instead of $(1/3)-$, and for $1/6 < s < 1/2$, it
is $2s$ instead of $(s - 1/6)-$; (iv) for any $0 \le s < 1/2$, the estimate is sharp.
\end{remark}
\begin{remark}
The  estimate  \eqref{smoothing 1 of u by w_L B} 
for $s >1/2$ answers a question, raised by Tzvetkov in \cite{T}, and  improves the estimate conjectured in \cite{T}.
\end{remark}
\subsection{Enhanced approximation of $u(t)$}
It turns out that an enhanced version $w_{L, \ast}(t)$ of $w_L(t)$ is obtained by replacing 
 for any $n \ge 1$ the frequency $n^2 - \langle u_0^2 | 1 \rangle$ in the $n$th summand in \eqref{def w_L} by the 
$n$th BO frequency $\omega_n \equiv \omega_n(u_0)$.
To define $\omega_n$, let us recall the definition of the Lax operator 
of \eqref{BO}, 
\begin{equation}\label{def L_u}
L_u := \frac{1}{i} \partial_x - T_u \, , 
\end{equation}
acting on the Hardy space
$H_+ \equiv H^0_+$ with domain $H^1_+$. Here $T_u$ denotes the Toeplitz operator with symbol $u$, given by
$$
T_u[f] := \Pi[u f] \, , \qquad \forall f \in H^1_+ \, .
$$
The operator $L_u$ is self-adjoint and bounded from below. Since it has a compact resolvent, its spectrum is discrete.
We list the eigenvalues of $L_u$ in increasing order and with their multiplicites, $\lambda_0(u) \le \lambda_1(u) \le \lambda_2(u) \le \cdots$. 
By \cite[Proposition 2.1]{GK} 
\begin{equation}\label{def gamma_n}
\gamma_n := \lambda_n - \lambda_{n-1} - 1 \ge 0 \, , \ \ \ \forall \, n \ge 1 \, ,
\end{equation}
and by \cite[Proposition 3.1]{GK}, the following trace formulas hold,
\begin{equation}\label{trace formulas}
\lambda_n = n -2\sum_{k \ge n+1} \gamma_k\, , \quad \forall n \ge 0\, , \qquad   \|u\|_0^2 = 2 \sum_{n \ge 1} n \gamma_n \, ,
\end{equation}
where for notational convenience, we (often) do not indicate the dependence of $\lambda_n$ ($n \ge 0$) and $\gamma_n$ ($n \ge 1)$
on $u$. In particular, all eigenvalues of $L_u$ are simple. 
The $n$th BO frequency is then given by (cf. \cite[formula (8.4)]{GK})
\begin{equation}\label{formula frequencies}
\omega_n = n^2 -  \la u_0^2 \, | \, 1 \ra + 2 \sum_{k > n} (k-n) \gamma_k \, .
\end{equation}
By \cite[Proposition 5]{GKT1},  
it follows that $ \omega_n(u_0)-(n^2 - \la u_0^2\vert 1\ra) = O(n^{-2s})$ as $n \to \infty$, uniformly 
on bounded subsets of initial data $u_0$ in $H^s_{r,0}$. 
Now we can define the enhanced approximation of $w$,
\begin{equation}\label{def w_L*}
w_{L, \ast}(t) : = \sum_{n \ge 1} e^{it \omega_n} \widehat w_0(n) e^{inx} \, .
\end{equation}
To state our enhanced approximation result, we introduce 
\begin{equation}\label{tau}
\tau(s) := \begin{cases}
1 \qquad \quad \  {\text{if}} \ \  s > 1/2 \\
1- \qquad   {\text{if}} \ \ s = 1/2 \\
s+\frac 12 \quad  \ \,   {\text{if}} \ \ 0 \leq s < 1/2
\end{cases} 
\end{equation}
\begin{theorem}\label{smoothing 2}
For any $u_0 \in H^s_{r,0}$ with $s\geq 0$ there exists $M_s> 0$ so that for any $t\in \R$,
\begin{equation}\label{approx 2}
  \| w(t) -  w_{L, \ast}(t) \|_{s+\tau (s)}  \le  M_s   \, , 
\end{equation}
\begin{equation}\label{smoothing 2 of u by w_L B}
u(t) = 2 {\rm Re} \big(  e^{i \partial_x^{-1} u(t)} i w_{L,\ast}(t)  \big) + r_\ast(t)\, ,\qquad 
 \| r_\ast(t) \|_{s +\tau (s)} \le  M_s \, .
\end{equation}
The constant $M_s > 0$ can be chosen uniformly for bounded subsets of initial data $u_0$ in $H^s_{r,0}$. 
\end{theorem}
\begin{remark}
Smoothing properties can also be proved for solutions in some Sobolev spaces $H^s_{r,0}$ with $s$ negative.
In order to limit the size of the paper, we decided to focus on solutions 
in $H^s_{r,0}$ with $s \ge 0$.
\end{remark}
Note that, in addition to providing a better gain of regularity for $s$ in the interval $0 \le s <1/2$, 
the estimates of Theorem \ref{smoothing 2} are {\em uniform in time}. 
These improvements  are obtained by taking into account that the BO equation is {\em integrable}  
and as a consequence that the BO dynamics are determined by the BO frequencies.

%
\subsection{High frequency approximation of the nonlinear Fourier transform}
A second key ingredient in the proof of the smoothing properties of solutions of \eqref{BO} is the high frequency approximation 
of the nonlinear Fourier transform $\Phi$ of the Benjamin-Ono equation, which was constructed in \cite{GK}, \cite{GKT1}.
Let us review the definition of $\Phi$ and the properties of $\Phi$ needed to state our smoothing results for solutions of \eqref{BO}.
To this end, we first need  to review further properties of  the Lax operator $L_u$, introduced in the previous subsection.
 It is shown in \cite{GK} that $L_u$ admits an orthonormal basis of eigenfunctions $f_n \equiv f_n(\cdot, u) \in H^1_+$, $n \ge 0$, 
uniquely determined by the normalisation conditions
\begin{equation}\label{normalisation f_n}
\langle f_0 | 1 \rangle > 0 \, , \qquad \langle f_n | e^{ix} f_{n-1} \rangle > 0 \, , \ \ \forall \, n \ge 1 \, .
\end{equation}
For any $s \in \R$, denote by $\h^{s} \equiv \h^{s} (\N, \C)$ the weighted $\ell^2$-sequence space
$$
\h^{s} := \{ z = (z_n)_{n \ge 1} \subset \C \, : \,  \|z\|_s < \infty \} \, , \qquad
\|z\|_s := ( \sum_{n \ge 1} n^{2s} |z_n|^2 )^{1/2} \, .
$$
In \cite{GK}, we introduced the map 
\begin{equation}\label{def Birkhoff}
\Phi : H^{0}_{r, 0} \to \h^{1/2} , \, u \mapsto (\zeta_n(u))_{n \ge 1} \, , \qquad \zeta_n(u) := \frac{\langle 1 | f_n(\cdot, u) \rangle}{\sqrt{\kappa_n(u)}} \,, 
\end{equation}
and proved that $\Phi$ is a homeomorphism and that
\begin{equation}\label{absolute value Birkhoff coordinates}
|  \zeta_n(u) |^2 = \gamma_n(u)\, , \qquad \forall \, n \ge 1 \, , \ \forall \, u \in H^0_{r,0} \, .
\end{equation}
 Here $\kappa_n \equiv \kappa_n(u) > 0$, $n \ge 1$,  are defined as absolutely convergent infinite products,
\begin{equation}\label{formula kappa_n}
\kappa_n = \frac{1}{\lambda_n - \lambda_0} \prod_{1\leq p \ne n} (1 - \frac{\gamma_p}{\lambda_p - \lambda_n}) \, .
\end{equation}
It is shown in \cite{GKT1} that for any $s \ge 0$, the restriction of $\Phi$ to $H^{s}_{r, 0}$ takes values in $\h^{s + 1/2}$
and in  \cite{GKT3}-\cite{GKT4} that 
\begin{equation}\label{restrictions of Phi}
\Phi : H^{s}_{r, 0} \to \h^{s + 1/2}
\end{equation} 
is a real analytic diffeomorphism.

One of the principal features of $\Phi$ is that it can be used to solve the initial value problem of \eqref{BO}.
Indeed, it is shown in \cite{GKT1} that for any initial data $u_0 \in H^s_{r, 0}$ with $s \ge 0$,  
the solution $t \mapsto u(t) \in H^s_{r, 0}$ of \eqref{BO} with initial data $u(0) = u_0$ satisfies
\begin{equation}\label{evolution in Birkhoff}
\Phi(u(t)) = ( e^{it \omega_n} \zeta_n(u_0))_{n \ge 1} \, , 
\end{equation}
where $\omega_n \equiv \omega_n(u_0)$, $n \ge 1$, denote the BO frequencies of $u_0$,
introduced in \eqref{formula frequencies} above.
The high frequency approximation of $\Phi :H^{0}_{r, 0} \to \h^{1/2}$ is then defined as the map $\Phi_0:H^{0}_{r, 0} \to \h^{ 1/2}$,
given by
\begin{equation}\label{approximation Phi_0}
\Phi_0(u):=  \big( \sqrt{n} \langle 1 | g_\infty e^{inx}  \rangle \big)_{n \ge 1} \, , \qquad   g_\infty \equiv g_\infty(\cdot, u):= e^{i \partial_x^{-1}u} \, .
\end{equation}
Note that $\Phi_0$ is a quasi-linear perturbation of the Fourier transform. 
Indeed, since 
$$
n \langle 1 \,  | \, g_\infty e^{inx}  \rangle =  \langle \overline{g_\infty} \,  | \, n e^{inx}  \rangle = \langle \overline{g_\infty} \, | \, \frac{1}{i} \partial_x e^{inx}  \rangle \, ,
$$ 
integration by parts yields
\begin{equation}\label{formula 1}
n \langle 1 \, | \, g_\infty e^{inx}  \rangle =  \langle \frac{1}{i} \partial_x \overline{g_\infty} \, | \, e^{inx}  \rangle  = - \langle u \overline{g_\infty} \, |  \, e^{inx}  \rangle \, .
\end{equation}
This shows that
\begin{equation}\label{quasilinear perturbation of Fourier transform}
\Phi_0(u) = ( - \frac{1}{\sqrt{n}} \langle u \, | \, g_\infty e^{inx} \rangle)_{n \ge 1} \, .
\end{equation}
Since $g_\infty$ is a function of $\partial_x^{-1} u$, the map $\Phi_0$ can be viewed, up to scaling, as 
a quasi-linear perturbation of the Fourier transform.

The following smoothing properties of
 $\Phi - \Phi_0$, which are of independent interest,
 are key ingredients in the proofs of Theorem \ref{smoothing 1} and Theorem \ref{smoothing 2}.
\begin{theorem}\label{Theorem Phi - Phi_L}
 For any $s \geq 0$, the map $\Phi - \Phi_0$ is smoothing of order $\tau (s)$ with $\tau(s)$ as defined in \eqref{tau}, i.e., $\Phi - \Phi_0$ is a continuous map
from $H^s_{r, 0}$ with values in $\h^{s+1/2 + \tau (s)}$.
Furthermore, there exists $u \in H^{1/2}_{r,0}$ with the property that $\Phi (u)-\Phi_0(u)\notin \h^2$
and similarly, for any $0 < s <1/2$, there exists $ u\in H^s_{r,0}$ so that for any $\e >0$,
$\Phi (u)-\Phi_0(u)\notin \h^{s+1/2+\tau (s)+\e}$.
\end{theorem}
\begin{remark} 
(i) Theorem \ref{Theorem Phi - Phi_L} says that $\Phi_0$ can be viewed as a {\em quasi-linear} high frequency approximation of $\Phi$.

\noindent
(ii) Note that for any $u \in H^0_{r,0}$, one has
$$
\mathcal G(u) = \partial_x \Pi(e^{-i\partial_x^{-1}u}) = 
\partial_x \Pi(\overline{ g_\infty})  = \sum_{n \ge 1}  in  \langle \overline{ g_\infty} | e^{inx}  \rangle e^{inx} 
$$
and hence
\begin{equation}\label{w and Phi_0}
 \Phi_0(u) = \big(-  \frac{i}{  \sqrt{n}} \langle \mathcal G(u) \,  | \, e^{inx} \rangle \big)_{n \ge 1} \, .
\end{equation}
It then follows from Theorem \ref{diffeogauge} in Appendix \ref{Appendix C} that for any $s \ge 0$, $\Phi_0: H^s_{r,0} \to \frak h^{s+1/2}$ is a diffeomorphism
onto an open proper subset of $\frak h^{s+1/2}$. 

\noindent
(iii) In \cite{Tao}, Tao asks whether the gauge transform $\mathcal G$
is related to the integrability of the Benjamin--Ono equation.  
In view of the formula \eqref{w and Phi_0} for $\Phi_0$, Theorem \ref{Theorem Phi - Phi_L} answers Tao's question 
for the BO equation on $\T$
by proving that (up to scaling) the Fourier transform of $\mathcal G$ is
a high frequency approximation of the Birkhoff map $\Phi$.
\end{remark}
\begin{remark}
In Appendix \ref{approximation dPhi}, we provide high frequency approximations of the differentials of $\Phi$ and of $\Phi^{-1}$.
Such approximations are useful when studying the pullback of vector fields by $\Phi$ or $\Phi^{-1}$.
\end{remark}

As a corollary of Theorem \ref{Theorem Phi - Phi_L}, we obtain smoothing properties of solutions of \eqref{BO}, 
expressed in the coordinates provided by $\Phi$. To this end,  we introduce the following evolution maps:
given any $\alpha \ge \frac 12$ and $t \in \R$, define 
for any initial data $\zeta \in \h^\alpha$, 
\begin{eqnarray*}
\mathcal S_L (t,\zeta) &:=& \big( e^{it(n^2 - 2\Vert \zeta\Vert_{1/2}^2)} \zeta_n \big)_{n \ge 1}  \in  \h^{\alpha} \, , \\
\mathcal S_{L,\ast }(t,\zeta) &:=& \big( e^{it(n^2 - 2\Vert \zeta\Vert_{1/2}^2+\delta_n(\zeta))} \zeta_n \big)_{n \ge 1}  \in  \h^{\alpha} \, ,
\end{eqnarray*} 
where
$$
 \delta_n(\zeta):=2\sum_{k>n}(k-n)|\zeta_k|^2 \,  .
$$
\begin{corollary}\label{smoothing of u by w_L}
 For any $u_0 \in H^{s}_{r,0}$ with $s\geq  0$, there exists $M_s > 0$ so that for any $t \in \R$,
\begin{equation}\label{estimate for Phi(u_0) for s > 1/2}
\|\Phi(\mathcal S(t,u_0)) - \mathcal S_L(t,\Phi_0(u_0)) \|_{s+\frac 12 +\sigma(s)} \le  M_s  \langle t \rangle \, ,  
\end{equation}
\begin{equation}\label{enhanced estimate for Phi(u_0) for s > 1/2}
\|\Phi(\mathcal S(t,u_0)) - \mathcal S_{L,\ast}(t,\Phi_0(u_0)) \|_{s+\frac 12 + \tau (s)} \le  M_s  \, .
\end{equation}
The constant $M_s > 0$ can be chosen uniformly on bounded subsets of initial data in $H^{s}_{r,0}$.
\end{corollary}

\subsection{Applications}
In Section \ref{BO on Hoelder}, we apply Theorem \ref{smoothing 1}  to study the action of the Benjamin--Ono flow $\mathcal S(t)$ on  the H\"older spaces $C^\alpha(\T, \R)$. 
In particular, we prove that there exists a subset $N\subset \R $ of Lebesgue measure $0$ so that for any $t\notin N$ and any $1/2 < \alpha < 1$, 
 $\mathcal S(t)$ does not map $\cap_{\e >0}C^{\alpha -\e}(\T, \R)$ into $\cup_{\e >0}C^{\alpha -1/2+\e} (\T, \R)$.
  On the other hand it is easy to check that for any $t\in \R$,  $\mathcal S(t)$ maps $\cap _{\e >0}C^{\alpha -\e}(\T, \R)$ into $\cap_{\e >0}C^{\alpha -1/2-\e}(\T, \R)$. 
  We refer to Section \ref{BO on Hoelder} for additional results.

\subsection{Comments}
(i) Birkhoff maps have been constructed for integrable PDEs such as the KdV equation (\cite{KP1}), the mKdV equation, and the defocusing NLS equation (\cite{GK1}). 
Each of these maps admits a high frequency approximation, similar to the one of the Birkhoff map of the Benjamin-Ono equation, stated in Theorem \ref{Theorem Phi - Phi_L}.
But in contrast to the Benjamin-Ono equation, it is given (up to scaling) by the Fourier transform -- see \cite{KST1} (cf. also \cite{KuPe}), \cite{KST2}, \cite{KST3}.
Hence for these equations, the Birkhoff map can be viewed as a {\em semilinear} perturbation of the Fourier transform. \\
(ii) Smoothing properties, similar to the ones stated in Theorem \ref{smoothing 1} and  Theorem \ref{smoothing 2} for the Benjamin-Ono equation,
have been established previously for solutions of integrable PDEs such as the KdV equation (\cite{ET1}, \cite{ET3}, \cite{KST1}) and the defocusing NLS equation 
(\cite{ET2}, \cite{ET3}, \cite{KST4}). In contrast to \eqref{smoothing 1 of u by w_L B}, \eqref{smoothing 2}, 
these smoothing properties are obtained by approximating
 solutions of these equations by solutions of the Airy equation (in the case of the KdV equation) and by solutions of the linear Schr\"odinger equation (in the case of the defocusing
NLS equation) or by enhanced versions of solutions of these linear equations, involving the KdV and NLS frequencies.

\subsection{Organization of the paper. Notations}
The paper is organized as follows. 
In Section 2 we prove Theorem \ref{Theorem Phi - Phi_L},
which then is used in Section 3 to derive   Theorems \ref{smoothing 1} and \ref{smoothing 2} and Corollary \ref{smoothing of u by w_L}. 
In Section 4, we apply Theorem \ref{smoothing 1} to study the action of the Benjamin--Ono flow on  H\"older spaces. 
In Appendix \ref{Hankel operators}, we record smoothing properties of Hankel operators, which are used throughout the main body of the paper.
In Appendix \ref{Appendix C}, we prove diffeomorphism properties of Tao's gauge transform.
Finally, in Appendix \ref{approximation dPhi}, we derive high frequency approximations of the differential of $\Phi$ and the one of $\Phi^{-1}$.

By and large, we will use the notation established in
\cite{GK}. In particular, 
the $H^s$-norm of an element $v$ in the 
Sobolev space $H^s_c \equiv H^s(\T, \C)$, $s \in \R$,
will be denoted by $\|v\|_s$. It is defined by 
\begin{equation}\label{Hs norm}
\|v\|_s = 
\big( \sum_{n \in \Z} \langle n \rangle^{2s}
|\widehat v(n)|^2 \big)^{1/2}  \, , \qquad  \langle n \rangle = \max\{1, |n|\}  \, .
\end{equation}
For $s=0$, we usually write $\|v\|$ for $\|v\|_0$. 
By $\langle \cdot \, | \, \cdot \rangle$, we will also
denote the extension of the $L^2$-inner product,
introduced in \eqref{L2 inner product}, 
to $ H^{-s}_c\times H^s_c$, $s \in \R$, by duality.
By $H_+ \equiv H^0_+$ we denote the Hardy space, consisting
of elements $f \in L^2(\T, \C) \equiv H^0_c$ with
the property that 
$ \widehat f(n) = 0$ for any $n < 0$.
More generally, for any $s \in \R$,  $H^{s}_+$
denotes the subspace of $H^s_c,$
consisting of elements $f \in H^s_c$ with the property
that $ \widehat f(n) = 0$ for any $n < 0$.
By $\h^{s} \equiv \h^{s} (\N, \C)$ we denote the weighted $\ell^2$-sequence space
$$
\h^{s} := \{ z = (z_n)_{n \ge 1} \subset \C \, : \,  \|z\|_s < \infty \} \, , \qquad
\|z\|_s := ( \sum_{n \ge 1} n^{2s} |z_n|^2 )^{1/2} \, .
$$
For notational convenience, we often write $z_n = \h^{s}_n$
for a sequence $(z_n)_{n \ge 1}$ in $\h^{s}$. The same notation is also used for other sequence spaces
such as $\ell^1 \equiv \ell^1(\N, \R)$.
Finally, for any $a$, $b$ in $\R$, the expression $a \lesssim b$ means that there exists $C > 0$ so that $a \le C b$.

\smallskip
\noindent
{\em Acknowledgement.} We would like to warmly thank Nikolay Tzvetkov for interesting discussions and for sharing with us his (unpublished) work on the smoothing
properties of solutions of the Benjamin-Ono equation, which is at the origin on this paper.


\section{Proof of Theorem \ref{Theorem Phi - Phi_L}}\label{smoothing property for Phi}
In this section  we prove Theorem \ref{Theorem Phi - Phi_L}.
Throughout this section we assume that $s \geq 0$. Recall that $\Phi$ denotes the Birkhoff map,  
$$
\Phi : H^{s}_{r, 0} \to \h^{s + 1/2} , \, u \mapsto (\zeta_n(u))_{n \ge 1} \, , \qquad \zeta_n(u) = \frac{\langle 1 | f_n(\cdot, u) \rangle}{\sqrt{\kappa_n(u)}}\ ,
$$ 
where $f_n \equiv f_n( \cdot, u),$ $n\ge 0$, is the orthonormal basis of eigenfunctions of the Lax operator $L_u$ (cf. \eqref{def L_u}), 
uniquely determined by the normalization conditions \eqref{normalisation f_n},
and $\kappa_n \equiv \kappa_n(u) > 0$ are scaling factors given by \eqref{formula kappa_n}.
To prove that  $\Phi_0(u) =  \big( \sqrt{n} \langle 1 | g_\infty e^{inx}  \rangle \big)_{n \ge 1}$, defined in \eqref{approximation Phi_0}, 
approximates $\Phi$, 
we introduce the auxiliary map
\begin{equation}\label{def Phi_1}
\Phi_1 : H^{0}_{r, 0} \to \h^{1/2} , \, u \mapsto  \big( \sqrt{n} \langle 1 | f_n  \rangle \big)_{n \ge 1}\, .
\end{equation}
Since $\lambda_n \langle 1 | f_n \rangle =  \langle 1 | L_u f_n \rangle =  - \langle u | f_n \rangle $, 
the map $\Phi_1$ can be viewed (up to scaling) as a version of the Fourier transform where the orthonormal basis $e^{inx}$, $n \ge 0$, of the Hardy space $H_+$
is replaced by the basis $f_n$, $n \ge 0$, of eigenfunctions of $L_u$.

In a first step we study the difference  $\Phi(u) - \Phi_1(u)$. Its $n$th component is given by
\begin{equation}\label{Phi - Phi_1}
 \zeta_n(u)  -  \sqrt{n} \langle 1 | f_n  \rangle = 
 \sqrt{n} ( \frac{1}{ \sqrt{ n \kappa_n}} - 1 )  \langle 1 | f_n  \rangle \, .
\end{equation}
We begin by deriving an estimate for $n \kappa_n$.
\begin{lemma}\label{smooting}
For any $u \in H^s_{r, 0}$ with $s \geq 0$, 
\begin{equation}\label{estimates n kappa_n}
n\kappa_n(u) = 1 + O(\frac{1}{n}) \  .
\end{equation}
As a consequence
\begin{equation}\label{smoothing of Phi_1}
 \frac{1}{\sqrt{n \kappa_n(u)}}  = 1 +  O(\frac{1}{n}) \  .
\end{equation}
For any given $s \ge 0$, the estimates for $n \kappa_n(u)$ and  $ \frac{1}{\sqrt{n \kappa_n(u)}} $
hold uniformly on bounded subsets of potentials $u$ in $H^s_{r,0}$.
\end{lemma}
\begin{proof}
In view of \eqref{formula kappa_n} we write $n\kappa_n - 1= I_n + II_n$ where
$$
I_n = ( \frac{n}{\lambda_n - \lambda_0} -1 ) \prod_{p \ne n} (1 - \frac{\gamma_p}{\lambda_p - \lambda_n}) \, , \qquad
II_n = \prod_{p \ne n} (1 - \frac{\gamma_p}{\lambda_p - \lambda_n}) - 1 \, .
$$
Let us first estimate $I_n$. Since for any $m \ge 0$, $\lambda_m \equiv \lambda_m(u)$ satisfies 
$\lambda_m = m - \sum_{k \ge m+1} \gamma_k$
(cf. \eqref{trace formulas})
one has 
\begin{equation}\label{formula  2 for eigenvalues}
\lambda_n - \lambda_0 = n +  \sum_{k =1}^n \gamma_k \, , \qquad  \forall \, n \ge 1 ,
\end{equation}
and in turn
$$
 \frac{n}{\lambda_n - \lambda_0} -1  = - \frac{1}{n} \,  \frac{\sum_{k =1}^n \gamma_k}{1 + \frac{1}{n}\sum_{k =1}^n \gamma_k} \, .
$$
The product $\prod_{p \ne n} (1 - \frac{\gamma_p}{\lambda_p - \lambda_n}) $ can be estimated as follows. 
Taking into account that
$$
| \prod_{p \ne n} (1 - \frac{\gamma_p}{\lambda_p - \lambda_n})  |  \le
\prod_{p \ne n} (1 + \frac{\gamma_p}{|\lambda_p - \lambda_n |}) =
\exp \big( \sum_{p \ne n} \log (1 + \frac{\gamma_p}{|\lambda_p - \lambda_n |}) \big)
$$
and that $0 \le \log (1 + a) \le a$ for any $a \ge 0$ one sees that
$$
| \prod_{p \ne n} (1 - \frac{\gamma_p}{\lambda_p - \lambda_n})  |  \le
\exp \big(  \sum_{p \ne n} \frac{\gamma_p}{|\lambda_p - \lambda_n |} \big)
\le \exp \big(  \sum_{p \ne n}  \gamma_p  \big),
$$
where for the latter inequality we used that $|\lambda_p - \lambda_n | \ge 1$ for any $p \ne n$.
By \eqref{trace formulas} it then follows that
\begin{equation}\label{estimate I_n}
I_n = O\big(\frac{1}{n}\big) \, .
\end{equation}
Next let us consider $II_n$. We start from the following identity,
\begin{equation}\label{idprod}
1-\prod_{p=1}^N (1-a_p)=\sum_{p=1}^N a_p\prod_{1\leq q<p}(1-a_q)\ ,
 \end{equation}
 which can be easily checked by induction on $N$
 for any given sequence of real (or complex) numbers $a_n$. 
 Note that \eqref{idprod} continues to hold for $N=\infty $ if the series of $a_p$ is absolutely convergent.
 Since $| 1 - a_p| \le 1 + |a_p| \le \exp(|a_p|)$ for any $p \ge 1$, one is led to the estimate
 \begin{equation}\label{estprod}
 \big| 1-\prod_{p=1}^\infty (1-a_p)\big| \leq \big(  \sum_{p=1}^\infty |a_p|     \big)\exp \big(  \sum_{p=1}^\infty |a_p|     \big) \,  .
 \end{equation}
 Using again that $|\lambda_p-\lambda_n|\geq 1$, it then follows that
 $$
 | II_n |\leq \big(\sum_{1\leq p\ne n}\frac{\gamma_p}{|\lambda_p-\lambda_n|}\big)\exp \big(\sum_{q=1}^\infty \gamma_q\big) \, .
 $$
 Since by \eqref{trace formulas} for any $p>n$,
 $$
 \lambda_p-\lambda_n =p-n+\sum_{n<k\leq p}\gamma_k\geq p-n \,  ,
 $$
 one infers that for any $p$, $|\lambda_p-\lambda_n|\geq |p-n|$. Hence
 $$
 \sum_{|p -n| \leq \frac n2}\frac{\gamma_p}{|\lambda_p-\lambda_n|}\leq \frac 2n \sum_{p}\gamma_p 
 $$
 and 
 $$
 \sum_{0 < |p - n| < \frac n2}\frac{\gamma_p}{|\lambda_p-\lambda_n|} \leq   \sum_{0 < |p - n| < \frac n2} \gamma_p
\le  \frac{2^{1+2s}}{n^{1+2s}} \sum_{p\ge 1}p^{1+2s}\gamma_p \,  .
 $$
 By \eqref{absolute value Birkhoff coordinates} and \cite[Proposition 5]{GKT1} it then follows that
\begin{equation}\label{estimate II_n}
II_n = O\big(\frac{1}{n}\big) \, , 
\end{equation}
which together with estimate \eqref{estimate I_n}  yields \eqref{estimates n kappa_n}.

By  \eqref{absolute value Birkhoff coordinates} and \cite[Proposition 5]{GKT1},
the estimate for $n \kappa_n(u)$  hold uniformly on bounded subsets of potentials in $H^s_{r,0}$ with $s \ge 0$.

To see that the estimate for $ \frac{1}{\sqrt{n \kappa_n(u)}} $ also holds uniformly on bounded subsets of potentials in $H^s_{r,0}$, 
it remains to find a uniform positive lower bound for $n \kappa_n$, $n \ge 1$, on such subsets. 
To this end note that by \eqref{formula kappa_n} and \eqref{formula 2 for eigenvalues}
$$
n \kappa_n  = \frac{1}{1 +\frac{1}{n} \sum_{k =1}^n \gamma_k}  \prod_{1 \le p < n} (1 + \frac{\gamma_p}{\lambda_n - \lambda_n})  \cdot
\prod_{p > n} (1 -  \frac{\gamma_p}{\lambda_p - \lambda_n} ) \, ,
$$
yielding, when combined with \eqref{trace formulas},
$$
n \kappa_n  \ge  \frac{1}{1 + | \lambda_0 | } \exp \big(- \sum_{p > n} -\log (1 - \frac{\gamma_p}{\lambda_p - \lambda_n}) \big) \, .
$$
Applying the estimate 
$$
-\log(1 - a) = \int_0^a \frac{1}{1-x} d x \le \frac{a}{1-a} \, , \qquad  \forall \, 0 < a < 1 \, ,
$$ 
to $a = \frac{\gamma_p}{\lambda_p - \lambda_n} $ and using that for any $p > n$,
$$
 \lambda_p-\lambda_n - \gamma_p = p - n + \sum_{n<k <  p}\gamma_k \ge 1 \, ,
$$
one concludes that for any $p > n$
$$
\frac{a}{1-a} =  \frac{\gamma_p}{\lambda_p - \lambda_n - \gamma_p} \le \gamma_p 
$$
and hence we obtain the following positive lower bound for $n \kappa_n$, $n \ge 1$,
$$
n \kappa_n  \ge  \frac{1}{1 + | \lambda_0 | } \exp \big(- \sum_{p > n} \gamma_p \big)      \ge   \frac{1}{1 + | \lambda_0 | } e^{ - |\lambda_0| } \, , 
\qquad \forall \, n \ge 1 \, .
$$
By \eqref{trace formulas}
the latter lower bound is uniformly bounded away from $0$ on bounded subsets of potentials $u$ in $H^0_{r,0}$.
\end{proof}
Combining \eqref{Phi - Phi_1} and \eqref{smoothing of Phi_1} then leads to the following
\begin{corollary}\label{estimate Phi - Phi1} 
For any $s \ge 0$, the difference $\Phi - \Phi_1$ is one-smoothing, meaning that it can be viewed as a continuous map
from $H^s_{r, 0}$ with values in $\h^{s+3/2}$. 
\end{corollary}
\begin{proof}
Going through the arguments of the proof of Lemma \ref{smooting} one sees that  $\Phi - \Phi_1 : H^s_{r, 0} \to \h^{s+3/2}$ is continuous
for any $s \ge 0$. \\
\end{proof}

Next we investigate $\Phi_1 - \Phi_0$. To this end we first derive asymptotic estimates for the scaling factors $\mu_n$, introduce in \cite[Section 4]{GK}. 
Recall that for any $n \ge 1$,
$0 < \mu_n \le 1$ is given by 
\begin{equation}\label{def mu_k}
\mu_n = \langle f_n | e^{ix}f_{n-1} \rangle^2
\end{equation} 
and admits the following infinite product representation (cf. \cite[(4.9)]{GK})
$$
\mu_n = (1 - \frac{\gamma_n}{\lambda_n - \lambda_0}) \prod_{p \ne n} (1 - b_{np})\, , \qquad
b_{np} = \gamma_n \frac{\gamma_p}{(\lambda_p - \lambda_n)(\lambda_{p-1} - \lambda_{n-1})} \, .
$$
\begin{lemma}\label{estimate mu_n}
 For any $u \in H^s_{r,0}$ with $s \geq 0$,
$$
0 \le 1 - \sqrt{\mu_n } \le  1 - \mu_n = \ell_n^{1, 2 + 2s }  \,  ,
$$
meaning that $(1 - \mu_n)_{n \ge 1} \in  \ell^{1, 2 + 2s}(\N, \R)$ (cf. 'Notation' in Section \ref{introduction}).
As a consequence
$$
0 \le \sqrt{ 1 - \sqrt{\mu_n }} \le  \sqrt{ 1 - \mu_n } =  \h_n^{1 + s }  \,  .
$$
For any given $s \ge 0$, these estimates hold uniformly on bounded subsets of potentials $u$ in $H^s_{r,0}$, $s \ge 0$.
\end{lemma}
\begin{proof} Using \eqref{estprod} and $|\lambda_p-\lambda_n|\geq |p-n|$, we have
$$
0\leq 1-\mu_n\leq S_n\exp(S_n) \, ,  \qquad
 S_n:=\frac{\gamma_n}{n} + \gamma_n\sum_{p \ne n} \frac{\gamma_p}{(p-n)^2}  \,  .
$$
Since by \eqref{absolute value Birkhoff coordinates} and \cite[Proposition 5]{GKT1},
$(\gamma_n)_{n \ge 1} \in \ell^{1, 1 + 2s}(\N, \R)$ and hence $(\frac{\gamma_n}{n})_{n \ge 1} \in \ell^{1, 2 + 2s}(\N, \R)$
it follows that
$$
\gamma_n\sum_{| p - n | > n/2} \frac{\gamma_p}{(p-n)^2} \le \frac{4 \gamma_n }{n^2} \sum_{p \ge 1} \gamma_p  =  \ell^{1, 3+ 2s}_n
$$
and 
$$
\begin{aligned}
 \gamma_n\sum_{0 < |p - n| \le n/2} \frac{\gamma_p}{(p-n)^2}  
& \le \gamma_n\sum_{0 < |p - n| \le n/2} \gamma_p \\
& \le \frac{ 2^{1+2s}\gamma_n }{n^{1 + 2s}} \sum_{p \ge 1} p^{1 + 2s}\gamma_p  =  \ell^{1, 2+ 4s}_n
\end{aligned} \, .
$$
 By \eqref{absolute value Birkhoff coordinates} and \cite[Proposition 5]{GKT1}, the stated estimates hold uniformly 
on bounded subsets of potentials $u$ in $H^s_{r,0}$ with $s \ge 0$.
\end{proof}

To estimate $\Phi_1 - \Phi_0$, introduce 
$\Xi :  H^s_{r,0} \to  \h^{s} , \, u \mapsto  (\Xi_n(u))_{n \ge 1}$ where
\begin{equation}\label{def Xi}
\Xi_n(u) :=  \sqrt{n }(\Phi_1(u) - \Phi_0(u))_n  = n  \langle 1 | f_n  \rangle -   n \langle 1 | g_\infty e^{inx}  \rangle  \, ,
\ \ \forall \, n \ge 1 \, .
\end{equation}
We recall that the exponent $\tau (s)$, $s > 0$, of the gain of regularity has been introduced in \eqref{tau}.
\begin{lemma}\label{estimate Xi}
Let $s \ge 0$. Then for any $u \in H^s_{r, 0}$,
$$
( \Xi_n(u))_{n \ge 1} \in  \h^{s + \tau (s)}  
$$
and  $( \Xi_n(u))_{n \ge 1}$ is uniformly bounded on bounded subsets of potentials $u$ in $H^s_{r,0}$.
Furthermore, there exists $u\in H^{1/2}_{r,0}$ with the property that $(\Xi_n(u))_{n \ge 1} \notin  \h^{1/2+1}$
and for $0 < s <1/2$, there exists 
$u\in H^s_{r,0}$ so that for any $\e >0$, $(\Xi_n(u))_{n \ge 1} \notin  \h^{s+\tau(s)+\e }$. 
\end{lemma}
\begin{proof}
Assume that $u \in H^s_{r,0}$ with $s \ge 0$. Then for any $n \ge 1$,
\begin{equation}\label{identity 1}
n  \langle 1 | f_n  \rangle = ( n - \lambda_n)  \langle 1 | f_n  \rangle + \langle 1 |  L_u f_n\rangle =  
( n - \lambda_n)  \langle 1 | f_n  \rangle -  \langle u | f_n  \rangle
\end{equation}
and by \eqref{formula 1}
\begin{equation}\label{identity 2}
n \langle 1 | g_\infty e^{inx}  \rangle = - \langle u | g_\infty e^{inx}  \rangle \, .
\end{equation}
Substituting \eqref{identity 1} and \eqref{identity 2} into the definition of $\Xi_n \equiv \Xi_n(u)$, one gets
$$
\Xi_n  = T_{1, n}  + \langle u | g_\infty e^{inx} - f_n \rangle \, , \qquad
T_{1, n}:=  ( n - \lambda_n)  \langle 1 | f_n  \rangle \,.
$$
We then write $ \langle u | \, g_\infty e^{inx} - f_n \rangle =T_{2,n}+T_{3,n}$ where
$$
T_{2,n} := \langle (\text{Id} - \Pi) u \,  | \,  g_\infty e^{inx} \rangle \, ,
\qquad T_{3,n} := \langle \Pi u \, | \,  g_\infty e^{inx} - f_n   \rangle \, .
$$
We thus have 
$$
\Xi_n(u)  = \sum_{j=1}^3 T_{j, n}(u) \, , \qquad T_{j, n} \equiv T_{j, n}(u) \, , \quad 1 \le j \le 3 \, .
$$
We begin by estimating $T_{1,n}$. By \eqref{trace formulas},
\begin{equation}\label{estimate n - lambda_n}
0 \le n - \lambda_n = \sum_{k \ge n+1} \gamma_k \le  \frac{1}{n^{1 + 2s}} \sum_{k \ge n+1} k^{1+2s}\gamma_k\, .
\end{equation}
Since by  \eqref{def Birkhoff}, Lemma \ref{smooting} , and \cite[Proposition 5]{GKT1}, one has 
\begin{equation}\label{estimate average f_n}
\langle 1 | f_n  \rangle = \h_n^{1+ s} \, ,
\end{equation}
the inequality  \eqref{estimate n - lambda_n} implies that 
\begin{equation}\label{estimate T1n}
T_{1,n}   =  ( n - \lambda_n)  \langle 1 | f_n  \rangle  = \h^{s + (2 + 2s)}_n \,  .
\end{equation}
To estimate $T_{2,n}$ we note that
\begin{equation}\label{def T_2,n}
T_{2,n}=\la   u  | \,  (\text{Id}  -\Pi )[g_\infty {\rm e}^{inx}]\ra =\la  \overline{g_\infty}  (\text{Id}  -\Pi )u | \,  {\rm e}^{inx}\ra \, .
\end{equation}
Hence for any $\rho \in \R$, $(T_{2,n})_{n \ge 1} \in \h^\rho$ if and only if $\Pi ( \overline g_\infty  (\text{Id} -\Pi )u )\in H^\rho _+$. 
If $u\in H^s_{r,0}$ with $s \ge 0$, then $g_\infty \in H^{s+1}_c$ and 
by the smoothing properties of Hankel operators recorded in Lemma \ref{Hankel}(i),(ii),(iii) with $\alpha = 1$,
we infer that
\begin{equation}\label{estimate T2n}
(T_{2, n}(u))_{n \ge 1} \in  \h^{s+\tau (s)} \,  , \qquad \forall \, u \in H^s_{r,0} \, .
\end{equation}
To estimate $T_{3, n}$, we write 
\begin{equation}\label{formula T_3,n}
T_{3,n}=\la \Pi u | \, e^{inx} (g_\infty -g_n) \ra \ , \qquad
g_n := f_n  \, e^{-inx} \, .
\end{equation}
By writing $g_\infty-g_n$ as a telescoping sum, 
\begin{equation}\label{tele}
g_\infty -  g_n = \sum_{k \ge n} ( g_{k+1} - g_k) \, , 
\end{equation}
we proved in \cite{GKT2} (cf. \cite[Proposition 9]{GKT2}) that for any $u \in H^s_{r,0}$ with $s \ge 0$, $g_n$
 converges in $H^{1+s}_c$ to $g_\infty$ as $n \to \infty$.
 To improve the estimates for $g_\infty -  g_n$, obtained in \cite{GKT2}, we write
 $$
 g_\infty -  g_n = \la g_\infty -g_n | \,  g_\infty \ra g_\infty + r_n \, .
 $$
Whereas $\la g_\infty -g_n | \,  g_\infty \ra$ will be estimated using \eqref{tele}, 
we need to analyze the remainder term $r_n \equiv r_n(u)$ further.
To this end we introduce
for any $u \in H^0_{r,0}$ and $n \ge 1$ the operators 
$$
K_n \equiv K_n(u) : \, H^{\frac 12 + \e}_c \to H^1_c, \, f \mapsto  g_\infty D^{-1}[\overline g_\infty \Pi _{<-n}(uf)] \, ,
$$
$$
K'_n \equiv K'_n(u): \, H_c^0 \to H^1_c, \, f \mapsto (n-\lambda_n)g_\infty D^{-1}[\overline{g_\infty} f] \, , \quad
$$
where $D^{-1}:=i\partial_x^{-1}$ and $\e > 0$.
\begin{lemma}\label{ginfty -gn}
For any $u \in H^0_{r,0}$ and $n\geq 1$,
\begin{equation}\label{eq:ginfty-gn}
(\text{Id} +K_n+K'_n)[g_\infty -g_n] = \la g_\infty -g_n | \,  g_\infty \ra g_\infty +K_ng_\infty \, .
\end{equation}
Furthermore, for any $u \in H^s_{r,0}$ with $0 \le s \le 1/2$, $0 < \e < 1/2+s$, and $n \ge 1$,
there exists a constant $C_{s,\e} > 0$ so that
\begin{align}
&\Vert K_nf \Vert_{1/2+\e}+\Vert K_n'f \Vert_{1/2+\e}\leq \frac{C_{s,\e}}{n^{1/2+s-\e}} \Vert f \Vert_{1/2+\e}\ , \label{est:K1}\\ 
&\Vert K_nf \Vert+ \Vert K_n'f \Vert \leq \frac{C_{s,\e}}{n^{1+s}}\Vert f \Vert_{1/2+\e}\ . \label{est:K2}
\end{align}
Similarly, for any  $u \in H^s_{r,0}$ with $s>1/2$, there exists a constant $C_{s} > 0$ so that
\begin{align}
&\Vert K_nf \Vert_{s}+\Vert K_n'f \Vert_{s}\leq \frac{C_{s}}{n} \Vert f \Vert_{s} \, , \label{est:K3}\\ 
&\Vert K_nf \Vert+ \Vert K_n'f \Vert \leq \frac{C_{s}}{n^{1+s}}\Vert f \Vert_{s} \,  . \label{est:K4}
\end{align}
The constants  $C_{s, \e}$ and $C_s$ can be chosen uniformly for $n \ge 1$ and
for bounded subsets of potentials $u \in H^s_{r,0}$.
\end{lemma}
\begin{proof}[Proof of Lemma \ref{ginfty -gn}]
Since for any $u \in H^0_{r,0}$ and  $n \ge 0$, $f_n  \in H^1_+$ is an eigenfunction of the Lax operator $L_u$ with eigenvalue $\lambda_n$,
$Df_n-\Pi [uf_n]=\lambda_n f_n$, the function $g_n= e^{-inx} f_n \in H^1_c$ satisfies
\begin{equation}\label{equation for g_n}
(D-u)g_n=(\lambda_n -n)g_n -\Pi _{<-n}[ug_n] \, ,
\end{equation}
where $\Pi_{<-n}$ is the $L^2$-orthogonal projector, defined by  
$$
f = \sum_{k \in \Z} \widehat f_k e^{ikx}  \mapsto  
\Pi_{<-n}[f] :=  \sum_{k < -n} \widehat f(k) e^{ikx} \, .
$$
Substracting \eqref{equation for g_n} from $(D-u)g_\infty =0$ we obtain
$$
(D-u)[g_\infty -g_n]=(n-\lambda_n)g_n +\Pi_{<-n}[ug_n] \, 
$$
or, after multiplying left and right hand side of the latter identity by $\overline{g_\infty}$,
and using that $D[ \overline{g_\infty} \, (g_\infty - g_n)] = - u \overline{g_\infty} \, (g_\infty - g_n) + \overline{g_\infty} \, D(g_\infty - g_n)$,
$$
D[\overline{g_\infty} \, (g_\infty -g_n) ]=(n-\lambda_n) \overline{g_\infty} \, g_n + \overline{g_\infty} \,\Pi_{<-n} [ug_n ] \,  .
$$
Applying $g_\infty D^{-1}$ to the left and right hand side of the latter identity, we obtain
$$
g_\infty -g_n  -  \la \, \overline{g_\infty} (g_\infty -g_n) | \, 1\ra g_\infty   = (K'_n+K_n)[g_n] \,  .
$$
Since $K'_n(g_\infty )=0$, identity \eqref{eq:ginfty-gn} follows.

Next we prove the estimates \eqref{est:K1} - \eqref{est:K2} for the operators $K_n$, $n \ge 1$.
First assume that $u \in H^s_{r,0}$ with $0 \le s \le 1/2$ and $0 < \e < 1/2+s$. 
Since $g_\infty \in H_c^{1+s}\subset H_c^{1/2+\e}$, 
it follows that for any $f \in H^{1/2 + \e}_c$,
\begin{eqnarray*}
\| K_nf\|_{1/2+\e}&=&\| g_\infty D^{-1}[\overline{ g_\infty} \, \Pi _{<-n}(uf)] \| _{1/2+\e} \\
& \lesssim &  \| \Pi _{<-n}[uf] \|_{-1/2+\e}\\
&\lesssim &\frac{1}{n^{1/2+s-\e}} \| uf \|_s 
\lesssim \frac{1}{n^{1/2+s-\e}} \| f \|_{1/2+\e}\ .
\end{eqnarray*}
Similarly,
\begin{eqnarray*}
\| K_nf \|&=&\| g_\infty D^{-1}[\overline{ g_\infty} \, \Pi _{<-n}(uf)] \|  
\lesssim   \| \Pi _{<-n}[uf] \|_{-1}\\
&\lesssim &\frac{1}{n^{1+s}} \| uf \|_s \lesssim \frac{1}{n^{1+s}} \| f \|_{1/2+\e}\ .
\end{eqnarray*}
Now assume $u \in H^s_{r,0}$ with $s>1/2$. Note that for such $s$, $H^s_c$ is an algebra, and $H^{1+s}_c$ acts on $H^{s-1}_c$ by mulitiplication. 
Hence for any $f \in H^s_c$, one gets
\begin{eqnarray*}
\| K_nf \|_{s} & = & \| g_\infty D^{-1}[\overline{g_\infty} \, \Pi _{<-n} (uf)] \| _{s} \\
&\lesssim & \| \Pi _{<-n}[uf] \|_{s -1}
\lesssim \frac{1}{n} \| uf \|_s 
\lesssim \frac{1}{n} \| f \|_{s}\ .
\end{eqnarray*}
Similarly,
\begin{eqnarray*}
\| K_nf \|&=&\| g_\infty D^{-1} [\overline{g_\infty} \,  \Pi _{<-n}(uf)] \| \\
& \lesssim  & \| \Pi _{<-n}[uf] \|_{-1}
\lesssim \frac{1}{n^{1+s}} \| uf \|_s 
\lesssim \frac{1}{n^{1+s}}\| f \|_{s}\ .
\end{eqnarray*}
Since for any $u \in H^s_{r,0}$ one has 
(cf. trace formula \eqref{trace formulas}, properties \eqref{absolute value Birkhoff coordinates},  \eqref{restrictions of Phi} of $\Phi$ )
$$
0 \le n-\lambda_n= 2 \sum_{k > n} \gamma_k \le  2 n^{-1-2s} \sum_{k > n} k^{1+ 2s}\gamma_k
\lesssim  \frac{1}{n^{1 + 2s} } ,
$$ 
the proofs of the claimed estimates for the operators $K'_n$, $n \ge 1$, are easier and 
hence we omit them.
Going through the arguments of the proof one verifies that the constants $C_{s, \e}$ and $C_s$
can be chosen uniformly for bounded subsets of potentials $u \in H^s_{r,0}$.
\end{proof}
Let us now continue with the proof of Lemma \ref{estimate Xi}.
The identity \eqref{eq:ginfty-gn}  and the estimates of the operators $K_n$ and $K_n'$ of Lemma \ref{ginfty -gn} allow
to write $g_\infty -g_n$ 
for $u \in H^s_{r,0}$ with $s \ge 0$ and 
$n$ sufficiently large 
as a Neumann series, $g_\infty -g_n=\e_n g_\infty + r_{n}$, where
$\e_n:=\la g_\infty -g_n\vert g_\infty \ra $ and 
\begin{equation}\label{neumann}
r_n : = (1-\e_n)\sum_{j=1}^\infty (-1)^{j+1} (K_n+K'_n)^j g_\infty \,  .
\end{equation}
Substituting $\e_n g_\infty + r_{n}$ for $g_\infty -g_n$ in the formula for $T_{3,n}$ one obtains
$T_{3,n} = \e_n \la \Pi u |  \, e^{inx} g_\infty \ra + (1 - \e_n ) \la \Pi u | \, e^{inx}  r_n \ra$, which we decompose further as
\begin{equation}\label{decom T_3,n}
T_{3,n} =   U_n  + V_n + W_n  \, , 
\end{equation}
where
\begin{equation}\label{def U_n}
U_n \equiv U_n(u) := \e_n \la \overline{g_\infty} \,  \Pi u | \,  e^{inx} \ra \, ,
\end{equation}
\begin{equation}\label{def V_n}
V_n \equiv V_n(u) :=   (1-\e_n)\sum_{r=2}^\infty (-1)^{r+1}\la \Pi u | \, e^{inx} (K_n+K'_n)^r [g_\infty] \ra \, ,
\end{equation}
and (using that $K_n'[g_\infty] = 0$)
\begin{equation}\label{def W_n}
W_n \equiv W_n(u) :=   (1-\e_n)\la \Pi u | \, e^{inx} K_n  [g_\infty]\ra\, .
\end{equation}

We first estimate $\e_n$. Representing $g_\infty-g_n$ by the telescoping sum \eqref{tele}, we get
$$ 
\e_n = \la g_\infty -g_n | \, g_\infty \ra =
\sum_{k\ge n}\la g_{k+1}-g_k | \,  g_\infty -g_k\ra +\sum_{k\ge n}(\la g_{k+1} | \,  g_k\ra -1) \,  .
$$
 Note that by the definition \eqref{def mu_k} of $\mu_{k+1}$ , 
 $$
 \la g_{k+1} | \,  g_k\ra=\sqrt{\mu_{k+1}} \,  ,
 $$
 and therefore, in view of Lemma \ref{estimate mu_n},
 \begin{equation}\label{est 1}
 \big| \sum_{k\ge n}(\la g_{k+1} | \, g_k\ra -1) \big| \lesssim  \frac{1}{n^{2+2s}} \, .
 \end{equation}
Moreover, $ \| g_{k+1} - g_k \|^2  =  \| f_{k+1} - Sf_k \|^2$ can be computed as
 $$
 \| g_{k+1} - g_k \|^2 
 =  2 - \langle g_{k+1} | \, g_k \rangle  - \langle g_{k} | \, g_{k+1}\rangle = 2 - 2 \sqrt{\mu_{k+1}} 
 $$
 and hence $\| g_{k+1} - g_k \| = \sqrt{2} (1 - \sqrt{\mu_{k+1}})^{1/2}$. 
By  Lemma \ref{estimate mu_n}, the Cauchy--Schwarz inequality, and the assumption
$s\geq 0$, one then infers that
 \begin{equation}\label{norm:ginfty -gn}
 \begin{aligned}
 \| g_\infty -g_n \| & \leq  \sum_{k \ge n} \| g_{k+1} -g_k \|  \le  \sum_{k\geq n}\sqrt 2(1-\sqrt{\mu_{k+1}})^{1/2} \\
 & \lesssim  \sum_{k\geq n} \big( k^{1+ s}(1-\sqrt{\mu_{k+1}})^{1/2} \big) \cdot \frac{1}{k^{1+ s}}
 \lesssim  \frac{1}{n^{(1+ 2s)/2 }} \, ,
 \end{aligned}
 \end{equation}
 and hence by the Cauchy-Schwarz inequality,
 \begin{align}\label{est 2}
&  \big| \sum_{k\ge n}\la g_{k+1}-g_k | \, g_\infty -g_k\ra \big|  \leq 
 \sum_{k\ge n}\sqrt 2(1-\sqrt{\mu_{k+1}})^{1/2} \| g_\infty -g_k \|  \nonumber \\
 & \lesssim  \sum_{k\ge n} \big( k^{1 + s}(1-\sqrt{\mu_{k+1}})^{1/2} \big) \cdot  \frac{1}{k^{1 + s +(1+ 2s)/2}}
 \lesssim \frac{1}{n^{1+2s  }}  \, .
\end{align}
Combining \eqref{est 1} and \eqref{est 2} we obtain 
\begin{equation}\label{est:epsilon}
| \e_n | \lesssim \frac{1}{n^{1+2s }} \,  .
\end{equation}
Using that $\overline{ g_\infty} \,  \Pi u \in H^s_c$ and taking into account the estimates  \eqref{est:epsilon}  
it then follows that $U_n$, defined by \eqref{def U_n}, satisfies
$$
U_n = \h_n^{s+1+2s} \, .
$$
Next we estimate $V_n$, defined by \eqref{def V_n}. 
First we consider the case where $0 \le  s \le 1/2$. 
From \eqref{est:K1} and \eqref{est:K2}, we have
$$
\begin{aligned}
|V_n| & \lesssim \sum_{r=2}^\infty \| (K_n+K'_n)^r [ g_\infty] \| \\
&\lesssim  \sum_{r=2}^\infty \frac{1}{n^{1+ s + (r-1)(1/2+s-\e)}} \| g_\infty \|_{1/2 + \e}
 \lesssim \frac{1}{n^{2s+3/2-\e}} \,  .
\end{aligned}
$$
Choosing $0 < \e < 1/4$,  we get
$$
V_n=\h_n^{2s+3/4} \,  .
$$
In the case where $s>1/2$, we use \eqref{est:K3} and \eqref{est:K4} to conclude that
\begin{eqnarray*}
|V_n|&\lesssim &\sum_{r=2}^\infty \| (K_n+K'_n)^r [g_\infty] \| \\
&\lesssim &  \sum_{r=2}^\infty \frac{1}{n^{1+s +r-1}}
\lesssim \frac{1}{n^s} \sum_{r=2}^\infty \frac{1}{n^r}
\lesssim \frac{1}{n^{s+2}} \,  .
\end{eqnarray*}
Hence for any $\e>0$,
$$V_n=\h_n^{s+3/2-\e} \, .
$$
It remains to estimate $W_n$, defined by \eqref{def W_n}, for $u \in H^s_{r,0}$ with $s\geq 0$.
First we note that
\begin{eqnarray*}
K_ng_\infty &=&g_\infty D^{-1}[\overline{g_\infty} \, \Pi_{<-n}(Dg_\infty )]
\\
&=& g_\infty D^{-1}D[\overline{ g_\infty} \cdot \Pi_{<-n}g_\infty ] - g_\infty D^{-1}[(D\overline{g_\infty}) \cdot \Pi_{<-n}g_\infty ]\\
&=&\Pi_{<-n}g_\infty - \|\Pi_{<-n}g_\infty  \|^2g_\infty +g_\infty D^{-1}[u\overline{g_\infty} \cdot \Pi_{<-n} g_\infty ]  \,  .
\end{eqnarray*}
Since $\la \Pi u | \, e^{inx} \Pi_{<-n}g_\infty \ra = 0$, 
it then follows that
$$
W_n  =  - \| \Pi_{<-n} g_\infty \|^2 \la \Pi u | \, e^{inx} g_\infty  \ra 
+ \la \overline{g_\infty} \,  \Pi u  | \, e^{inx} D^{-1}[u\overline{g_\infty} \,  \Pi_{<-n}g_\infty ]  \ra 
$$
or
\begin{equation}\label{first est W_n}
W_n = \h_n^{2+2s + s}+\la \overline{g_\infty} \,  \Pi u | \, e^{inx} D^{-1}[u \overline{g_\infty} \,  \Pi_{<-n} g_\infty ]\ra \,  .
\end{equation}
To estimate the latter term, let
$$
f_1:=\overline{ g_\infty} \, \Pi u \in H^s_c \, , \qquad f_2:=\overline{g_\infty} \,  u\in H^s_c \, . 
$$
Then
\begin{eqnarray*}
&& \ \ | \la \overline{g_\infty} \, \Pi u | \, e^{inx} D^{-1} [u \overline{g_\infty} \,  \Pi_{<-n} g_\infty ]\ra | ^2 \\
&& =\big| \sum_{k\in \Z \setminus \{n\}}      \overline {\widehat f_1(k)} \frac{1}{k-n}\sum_{p=1}^\infty \widehat f_2(k+p) \widehat g_\infty (-n-p)  \big|^2\\
&&\leq  \| f_1 \|_s^2 \sum_{k\ne n}\frac{1}{(k-n)^2\la k\ra ^{2s}}  \big|\sum_{p=1}^\infty \widehat f_2(k+p) \widehat g_\infty (-n-p)\big|^2
\end{eqnarray*}
Since by the Cauchy-Schwarz inequality,
$$
\begin{aligned}
& \ \ \big|\sum_{p=1}^\infty \widehat f_2(k+p) \widehat g_\infty (-n-p)\big|^2 \\
& \le \big(\sum_{p=1}^\infty | \widehat f_2(k+p) |^2 \la k+ p \ra^{2s} \big) \cdot 
 \big(\sum_{p=1}^\infty  \frac{1}{ \la k+ p \ra^{2s}} | \widehat g_\infty (-n-p) |^2 \big) \\
& \le \| f_2 \|_s^2  \big(\sum_{p=1}^\infty  \frac{1}{ \la k+ p \ra^{2s}} | \widehat g_\infty (-n-p) |^2 \big) \, ,
 \end{aligned}
$$
we get
\begin{align}\label{second est W_n}
& \ \ | \la \overline{g_\infty} \, \Pi u | \, e^{inx} D^{-1} [u \overline{g_\infty} \,  \Pi_{<-n} g_\infty ]\ra | ^2  \nonumber\\
& \leq  \| f_1 \|_s^2\, \| f_2 \|_s^2\, \sum_{p=1}^\infty B_{n,p} | \widehat g_\infty (-n-p) | ^2 \,  ,
\end{align}
where
$$
B_{n,p} := \sum_{k\ne n}\frac{1}{(k-n)^2\la k\ra ^{2s}\la k+p\ra ^{2s}}  \,  .
$$
Splitting the latter sum into two sums with domains $0<|k-n|<n/2$ and $|k-n|\geq n/2$ respectively, one concludes that for any $\e >0$,
$$
 \sum_{0 < |k - n| < n/2} \frac{1}{(k-n)^2\la k\ra ^{2s}\la k+p\ra ^{2s}}  \lesssim  \frac{1}{\la n\ra ^{2s} ( n + p) ^{2s}} 
  \sum_{0 < |k - n| < n/2} \frac{1}{(k-n)^2}\, ,
$$
$$
\begin{aligned}
 \sum_{ |k - n| \ge  n/2} \frac{1}{(k-n)^2\la k\ra ^{2s}\la k+p\ra ^{2s}}  
 & \lesssim \frac{1}{n^2}  \sum_{ |k| <  3n/2} \frac{1}{\la k \ra ^{2s} } 
 + \frac{1}{n^{1 - \e}}  \sum_{ |k| \ge  3n/2} \frac{1}{\la k \ra ^{1 + \e + 4s} } \, ,
\end{aligned}
$$
and hence
\begin{equation}\label{third est W_n}
B_{n,p}\lesssim  \frac{1}{n^{2s}(n+p)^{2s}} + \frac{1}{n^{1 +2s - \e } } \, .
\end{equation}
Given any $\gamma >0$, it then follows from \eqref{first est W_n}--\eqref{third est W_n}
that $\sum_{n=1}^\infty n^{2(s+\gamma)}|W_n|^2$ can be bounded by
\begin{eqnarray*}
&\lesssim& \sum_{n=1}^\infty  n^{2\gamma -4-4s} 
 \, \ell^1_n \, + \, \sum_{n=1}^\infty n^{2\gamma}\sum_{q>n}|\widehat g_\infty(-q)|^2(q^{-2s}+ n^{-1+\e }) \\
& \lesssim &\sum_{n=1}^\infty n^{2\gamma -4-4s} \, \ell^1_n \, + \, \sum_{q=1}^\infty (q^{2\gamma +1-2s}+q^{2\gamma+\e } )|\widehat g_\infty(-q)|^2 \, .
\end{eqnarray*}
 Using that $g_\infty \in H^{1+s}_c$ and choosing $\gamma$ as
 $$
 \tau_1(s) = \begin{cases}
 (1 + s)-   \quad   \mbox{ if } \, s > \frac 12 \\
 \frac 32 -  \qquad \quad \  \mbox{ if } \, s = \frac 12 \\
   \frac 12+2s \qquad  \  \mbox{ if } \, 0 \le s < \frac 12
 \end{cases} 
 $$
 we conclude that for any $u \in H^s_{r,0}$ the latter two sums are finite.

 In summary, we have proved that for any $u \in H^s_{r,0}$ with $s\geq 0$, $T_{3,n} =  \h_n^{s+\tau_1(s)}$.
 Note that by the definition \eqref{tau} of $\tau(s)$ and the one of $\tau_1(s)$, one has 
 $$
 \tau(0) = \frac 12 = \tau_1(0) \, ,  \qquad 
 \tau(s) < \tau_1(s) \, , \quad  \forall \, s > 0 \, .
 $$
Combining the estimate \eqref{estimate T1n} of $T_{1,n}$ and the estimate \eqref{estimate T2n} of
$T_{2,n}$ with the estimate of $T_{3,n}$, we conclude that for any $u \in H^s_{r,0}$, 
$$
\Xi_n(u) = \sum_{1 \le j \le 3} T_{j,n}(u) = \h_n^{s+\tau (s)}
$$ 
and that for any $u \in H^s_{r,0}$ with $s>0$, 
\begin{equation}\label{remainder1}
(\Xi_n(u)-T_{2,n}(u))_{n \ge 1} \in \h^{s+\tau_1(s)} \, , \qquad  \tau_1(s)>\tau (s) \, .
\end{equation}
Furthermore, going through the arguments of the proof one verifies that for any $s \geq 0$, the stated estimates hold uniformly 
on bounded subsets of potentials $u$ in $H^s_{r,0}$.

\smallskip
 It remains to study the optimality of these estimates for $0 < s \le 1/2$. 
 First consider the case $0 < s < 1/2$. Assume that $u\in H^s_{r,0}$ satisfies
 $( \Xi_n(u) )_{n \ge 1} = \h^{s+\tau(s)+\e}$ for some $\e>0$.
 In view of the estimate \eqref{remainder1} of $(\Xi_n(u)-T_{2,n})_{n \ge 1}$, it then follows that for $\e>0$ small enough,  
 $(T_{2,n}(u))_{n \ge 1}  \in \h^{2s+1/2+\e}$ or, by the formula \eqref{def T_2,n} for $T_{2, n}$,
 $$
 \Pi (\overline g_\infty (\text{Id}-\Pi)u)\in H^{2s+1/2+\e} \, .
 $$
 Set $v:=-\Pi u$. Since $u$ is real valued and has vanishing mean, one has $(\text{Id}-\Pi)u = - \overline v$ and hence
 $$
 \Pi [\overline g_\infty (\text{Id}-\Pi)u) ]=-\Pi [e^{- i \partial_x^{-1}\overline v} e^{- i \partial_x^{-1}v} \overline v] \,  .
 $$
 Since $e^{ -i \partial_x^{-1} v}$ is in the Hardy space $H^{s+1}_+$ and
 $e^{- i \partial_x^{-1}\overline v}$ in $H^{s+1}_-$, it follows that
 $$
  \Pi [e^{- i \partial_x^{-1}\overline v} ( \mbox{Id} - \Pi ) (e^{- i \partial_x^{-1}v} \overline v )] = 0
 $$
 and hence
 $$
 \begin{aligned}
  \Pi [\overline g_\infty (\text{Id}-\Pi)u)] & = -   \Pi [e^{ -i \partial_x^{-1}\overline v} \Pi (e^{- i \partial_x^{-1}v} \overline v )] \\
 & = - T_{e^{ -i \partial_x^{-1}\overline v}} [ \Pi (e^{- i \partial_x^{-1}v} \overline v )] \, .
 \end{aligned}
 $$
 Note that the Toeplitz operator $ T_{e^{ -i \partial_x^{-1}\overline v}}$ with symbol $e^{ -i \partial_x^{-1}\overline v}$
 is a linear isomorphism on $H^\rho _+$ for any $0 \le \rho \le 1+s$ and that its inverse is given by 
 $ T_{e^{ i \partial_x^{-1}\overline v}}$ (cf. e.g. \cite{La}, \cite[Section 6]{GKT4}). 
 Since  $1+s\geq 1/2+2s+\e $ for $\e $ small enough, we then conclude that
 \begin{equation}\label{Hankelcritique}
 \Pi [e^{-i \partial_x^{-1}v}\overline v ]\in H^{2s+1/2+\e} \, .
 \end{equation}
 Let us choose $u=-v-\overline v$ with
 \begin{equation}\label{example 1}
 v(x)=\sum_{k=1}^\infty \frac{e^{ikx}}{k^{1/2+s}\log (1+k)} \, .
 \end{equation}
 Clearly, $v\in H^s_+$ and $u\in H^s_{r,0}$. We claim that \eqref{Hankelcritique} fails for every $\e>0$. 
 Indeed, observe that the Fourier coefficients of $v$ and of $i \partial_x^{-1}v$ are positive, 
 and so are the Fourier coefficients of $(i\partial_x^{-1}v)^p \overline v $ for every integer $p\ge 1$. Expanding
 $$
 e^{i \partial_x^{-1}v}=\sum_{p=0}^\infty \frac{(i \partial_x^{-1}v)^p}{p!} \, ,
 $$
 we infer that the $k$th Fourier coefficient of 
 $\Pi (e^{i\partial_x^{-1}v}\overline v)$ is larger than the $k$th Fourier coefficient of $\Pi ((i \partial_x^{-1}v) \overline v)$, 
 and hence \eqref{Hankelcritique} implies
 \begin{equation}\label{Hankelcritiquebis}
 \Pi [(i \partial_x^{-1}v)\overline v] \in H^{2s+1/2+\e} \,  .
 \end{equation}
 For any $k\ge 1$, the $k$th Fourier coefficient of $\Pi [(i \partial_x^{-1}v)\overline v]$ can be computed as
 $$
 a_k = \sum_{j \in \Z} \frac{\widehat v(k - j ) \widehat{\overline v(j)}}{k-j}= 
 \sum_{j \le -1} \frac{\widehat v(k - j ) \widehat{ v(- j)}}{k-j} = 
 \sum_{\ell =1}^\infty \frac{\widehat v(k+\ell )\widehat v(\ell )}{k+\ell}
 $$
 and thus by the definition \eqref{example 1} of $v$, one has
 \begin{eqnarray*}
 a_k&=&\sum_{\ell =1}^\infty \frac{1}{(k+\ell)^{3/2+s}\log (1+k+\ell)\ell^{1/2+s}\log (1+\ell)}\\
 &\geq &\sum_{\ell =1}^k \frac{1}{(k+\ell)^{3/2+s}\log (1+k+\ell)\ell^{1/2+s}\log (1+\ell)}\\
 &\geq & \frac{2^{-3/2 -s}}{k^{1+2s}\log (1+2k)\log (1+k)} \, ,
  \end{eqnarray*}
  so that
  $$
  k^{2s+1/2+\e}a_k\geq b_k := \frac{2^{-3/2 -s }}{k^{1/2-\e}[\log (1+2k)]^2} \, .
  $$
Clearly,  for any $\e >0$, $(b_k)_{k \ge 1}$ is not an $\ell ^2$-sequence. This contradicts \eqref{Hankelcritiquebis}
and shows that for $u = - v - \overline{v}$ with $v$ given by \eqref{example 1}, 
 $(\Xi_n(u))_{n \ge 1} \notin \h^{2s+1/2+\e}$ for any $\e > 0$.

In the case $s=1/2$, we argue similarly as in the case $0 < s < 1/2$. Choose $u=-v-\overline v$ with 
$$ 
v(x)=\sum_{k=1}^\infty \frac{e^{ikx}}{k[\log (1+k)]^\alpha } \, ,
$$
where $1/2 < \alpha < 3/4$. Since $1/2 < \alpha$ it follows that $u\in H^{1/2}_{r,0}$. In this case,
the $k$th Fourier coefficient of $\Pi [i\partial_x^{-1}v \overline v] $ is
 \begin{eqnarray*}
 a_k&=&\sum_{\ell =1}^\infty \frac{\hat v(k+\ell )\hat v(\ell )}{k+\ell}=\sum_{\ell =1}^\infty \frac{1}{(k+\ell)^2[\log (1+k+\ell)]^\alpha \ell[\log (1+\ell)]^\alpha }\\
 &\geq &\sum_{\ell =1}^k \frac{1}{(k+\ell)^2[\log (1+k+\ell)]^\alpha \ell[\log (1+\ell)]^\alpha }\\
&\geq & \frac{4^{-1}}{k^{2}[\log (1+2k)]^\alpha [\log (1+k)]^\alpha }\sum_{\ell=1}^k\frac{1}{\ell}\\ 
&\gtrsim & \frac{1}{k^2[\log (k)]^{2\alpha -1}}
  \end{eqnarray*}
  so that
  $$
  k^{3/2}a_k\gtrsim  b_k := \frac{1}{k^{1/2}[\log (k)]^{2\alpha-1}} \, .
  $$
Since $\alpha <3/4$, $(b_k)_{k \ge 1}$ is not an $\ell^2$-sequence, 
hence $\Pi [(i\partial_x^{-1}v) \overline v] $ does not belong to $H^{3/2}$, 
and consequently $(\Xi_n(u))_{n \ge 1}\notin \h^{3/2}$. 

This finishes the proof of Lemma \ref{estimate Xi}.
\end{proof} 

 In view of the definition \eqref{def Xi} of $\Xi$, Lemma \ref{estimate Xi} yields the following 
\begin{corollary}\label{Phi_1 - Phi_0}
For any $s \geq 0$, the difference $\Phi_1 - \Phi_0$ is $\tau (s)$-smoothing, meaning that  for any $s \ge 0$, 
$\Phi_1 - \Phi_0$ is a continuous map from $H^s_{r, 0}$ with values in $\h^{s+1/2 + \tau (s)}$. 
Furthermore, there exists $u\in H^{1/2}_{r,0}$ so that $\Phi_1(u)-\Phi_0(u)\notin \h^{2}$
and for $0 < s < 1/2$, there exists $u\in H^s_{r,0}$ so that for any $\e >0$, 
$\Phi_1(u)-\Phi_0(u)\notin \h^{2s+1+\e} $. 
 \end{corollary}
\begin{proof}[Proof of Theorem  \ref{Theorem Phi - Phi_L}]
The claimed results directly follow from 
Corollary \ref{estimate Phi - Phi1}  and Corollary \ref{Phi_1 - Phi_0}.
\end{proof}


\section{Approximations of the Benjamin--Ono flow}\label{applications}
In this section we apply Theorem \ref{Theorem Phi - Phi_L} to prove
smoothing properties of the flow map of the Benjamin-Ono equation, 
stated in Theorem \ref{smoothing 1}, Theorem \ref{smoothing 2} and Corollary \ref{smoothing of u by w_L}.  

Recall from Section \ref{introduction} that for any $u_0 \in H^s_{r,0}$ with $s \ge 0$, we denote by $u(t) = \mathcal S(t,u_0)$
the solution of the Benjamin-Ono equation constructed in \cite{GKT1} and by $w(t)$
the gauge transformation of $u(t)$ (cf. \eqref{gauge}),
$$
w(t) = \mathcal G (u(t))=\partial_x \Pi(e^{-i\partial_x^{-1}u(t)}) \, , 
\quad w_0:= w(0) = \partial_x \Pi(e^{- i\partial_x^{-1}u_0}) \, .
$$
Furthermore, we introduced (cf. \eqref{def w_L}, \eqref{def w_L*})
\begin{equation}\label{def w_Lbis}
w_L(t) = \sum_{n \ge 1} e^{it(n^2 - \langle u_0^2 | 1 \rangle)}\widehat w_0(n) e^{inx}  \, ,
\end{equation}
\begin{equation}\label{def w_L*bis}
w_{L, \ast}(t) = \sum_{n \ge 1} e^{it \omega_n } \widehat w_0(n) e^{inx} \, . \qquad \
\end{equation}
\begin{proof}[Proof of Theorem \ref{smoothing 2}]
First we prove \eqref{approx 2}, saying that
for any bounded subset $\mathcal B$ of $H^s_{r,0}$ with $s \ge 0$, there exists $M_s > 0$ so that
\begin{equation}\label{approx wL*}
\sup_{t\in \R}\Vert w(t)-w_{L,\ast}(t)\Vert_{s+\tau (s)}\leq M_s \,  , \qquad  \forall \, u_0 \in \mathcal B \, .
\end{equation}
Recall that by \cite[Corollary 8, Appendix A]{GKT1}, 
there exists a bounded subset $\tilde {\mathcal B} $ of $H^s_{r,0}$ 
so that $\mathcal S(t, u_0)\in \tilde {\mathcal B}$ for any $t \in \R$ and $u_0\in \mathcal B$. 
By Theorem \ref{Theorem Phi - Phi_L}, applied to $u=u(t)$ and \eqref{w and Phi_0} it follows that
\begin{equation}\label{evolution w(t) vs zeta(t)}
\widehat {w(t)}(n)=i\sqrt n\zeta_n(u(t))+\rho_n(t) \,  ,  \qquad  \forall \, n\geq 1 \, ,
\end{equation}
where $(\rho_n(t))_{n\geq 1}$ is uniformly bounded in $\h^{s+\tau (s)}$ with respect to $t\in \R$. 
Since by \eqref{evolution in Birkhoff},
$$
i\sqrt n \zeta_n(u(t)) = i\sqrt n e^{it\omega_n}\zeta_n (u_0) \,  ,
$$
and by \eqref{evolution w(t) vs zeta(t)}  for $t=0$, 
$$
 \widehat {w(0)}(n) = i\sqrt n\zeta_n(u_0) + \rho_n(0) \, ,
$$
we conclude that for any $t \in \R$, 
\begin{equation}\label{approx wL*fourier}
\widehat {w(t)}(n) - e^{it\omega_n}\widehat {w(0)}(n) = \rho_n(t)-e^{it\omega_n}\rho_n(0) \,  ,
\qquad \forall n\geq 1 \, ,
\end{equation}
which proves \eqref{approx wL*}.
It remains to prove estimate \eqref{smoothing 2 of u by w_L B}. 
Following the suggestion in \cite{T}, we write
$u(t) = e^{i\partial_x^{-1} u(t)} e^{- i\partial_x^{-1} u(t)} u(t)$ as
$$
u(t) = e^{i\partial_x^{-1} u(t)}  \Pi [e^{- i\partial_x^{-1} u(t)} u(t)]
+ e^{i\partial_x^{-1} u(t)}  (\text{Id} - \Pi) [e^{- i\partial_x^{-1} u(t)} u(t)] \, .
$$
Since $\Pi [e^{- i\partial_x^{-1} u(t)} u(t)] = i w(t)$ one has 
$$
e^{i\partial_x^{-1} u(t)}  \Pi [e^{- i\partial_x^{-1} u(t)} u(t) ]
= e^{i\partial_x^{-1} u(t)}  iw_{L,\ast}(t) + e^{i\partial_x^{-1} u(t)}  i(w(t) - w_{L,\ast}(t))\, .
$$
Recall that $u(t)$ is real valued and satisfies $\la u(t) | 1 \ra = 0$. Hence $u(t) = 2 \text{Re} (\Pi u(t)) $.
Splitting the term $2 \text{Re} \big( \Pi[ e^{i\partial_x^{-1} u(t)}  iw_{L,\ast}(t) ] \big)$ as
$$
2 \text{Re}( e^{i\partial_x^{-1} u(t)}  iw_{L,\ast}(t) )
- 2 \text{Re}\big(  (\text{Id} - \Pi)[ e^{i\partial_x^{-1} u(t)}  iw_{L,\ast}(t) ] \big)
$$ 
one then concludes that
$$
u(t) = 2\text{Re} \big(e^{i\partial_x^{-1} u(t)}  iw_{L,\ast}(t)  \big) + r_\ast (t)
$$
where $r_\ast (t):=  I(t) + II(t) + III(t)$ and 
$$
I(t) :=  - 2\text{Re} \big((\text{Id} - \Pi) [ e^{i\partial_x^{-1} u(t)}  iw_{L,\ast}(t) ] \big) \, , \qquad
$$
$$
II(t) := 2\text{Re} \big(\Pi [e^{i\partial_x^{-1} u(t)}  i(w(t) - w_{L,\ast}(t)) ] \big) \, , \qquad \
$$
$$
III(t) := 2 \text{Re} \big(  \Pi[ e^{i\partial_x^{-1} u(t)}  (\text{Id} - \Pi) (e^{- i\partial_x^{-1} u(t)} u(t))]   \big) \, .
$$
Since $\partial_x^{-1} u(t) \in H^{s+1}_{r, 0}$, the claimed estimate of $r_*(t)$ is obtained by estimating the term $II(t)$  
with the help of estimate \eqref{approx 2}
and the terms $III(t)$, $I(t)$ by Lemma \ref{Hankel} in Appendix A. 
\end{proof}
\begin{proof}[Proof of Theorem  \ref{smoothing 1}]
First we prove estimate  \eqref{approx 1}, saying that
for any bounded subset $\mathcal B$ of $H^s_{r,0}$ with $s \ge 0$, there exists $M_s > 0$ so that
\begin{equation}\label{approx wL}
\sup_{t\in \R}\Vert w(t)-w_{L}(t)\Vert_{s+\sigma (s)}\leq M_s \la t \ra  \,  , \qquad  \forall \, u_0 \in \mathcal B \, .
\end{equation}
To this end note that by \eqref{formula frequencies},
$$
|e^{it\omega_n}-e^{it(n^2-\la u_0^2\vert 1\ra )}| \leq 
2|t| \sum_{k>n}(k-n) \gamma_k(u_0) 
\leq \frac{C|t|}{n^{2s}} \, ,
$$
where the constant $C > 0$ can be chosen uniformly for $u_0\in \mathcal B$. 
Combined with \eqref{approx wL*fourier}, estimate  \eqref{approx wL} follows.
Estimate \eqref{smoothing 1 of u by w_L B} can be proved in a similar way 
as the estimate \eqref{smoothing 2 of u by w_L B} 
in the proof of Theorem \ref{smoothing 2} and hence we omit the details. 

It remains to prove the optimality statement of Theorem \ref{smoothing 1}. Let $0 < s < 1/2$, and consider $u_0\in H^s_{r,0}$ with the property that
\begin{equation}\label{gammacritique}
\gamma_n(u_0)=\frac{1}{n^{2+2s}[\log(1+n)]^2} \, , \qquad \forall \, n\geq 1 \, .  
\end{equation}
(Such $u_0$ exist since by \cite[Proposition 5 in Appendix A]{GKT1}, $\Phi : H^s_{r,0}\to \h^{s+1/2}$ is onto.)
By the definition \eqref{sigma}, one has $\sigma(s) = 2s$.
Assume that there exist $t\ne 0$ and $\e >0$ so that 
\begin{equation}\label{assumption 1}
w(t)-w_L(t)\in H^{3s+\e}_+ \, . 
\end{equation}
Since $3s < 2s+1/2 = s+\tau (s)$,
estimate \eqref{approx 2} of Theorem \ref{smoothing 2} then implies that for $\e$ sufficiently small, 
\eqref{assumption 1} is equivalent to
\begin{equation}\label{to be contradicted}
w_{L,\ast}(t)-w_L(t)\in H^{3s+\e} \, . 
\end{equation}
By the above formulas \eqref{def w_Lbis}, \eqref{def w_L*bis}, this is equivalent to
$$
|e^{it\omega_n}-e^{it(n^2-\la u_0^2\vert 1\ra )}| \cdot |\widehat w_0(n)| = \h_n^{3s+\e} 
$$
(recall that $w_0 = \partial_x \Pi(e^{- i\partial_x^{-1}u_0}) $), or
\begin{equation}\label{estimate for w_0}
\big(\sum_{k>n}(k-n) \gamma_k(u_0) \big) |\widehat w_0(n)| =\h_n^{3s+\e} \,  ,
\end{equation}
where we used that for $n$ sufficiently large (cf. \eqref{formula frequencies}),
$$
\begin{aligned}
|e^{it\omega_n}-e^{it(n^2-\la u_0^2\vert 1\ra )}|  &= 2 \sum_{k>n}(k-n) \gamma_k(u_0) \,  \big| \int_0^t e^{is 2\sum_{k>n}(k-n) \gamma_k(u_0)} d s \big| \\
& \sim |t|  \sum_{k>n}(k-n) \gamma_k(u_0) \, .
\end{aligned}
$$
Since by \eqref{gammacritique},
\begin{eqnarray*}
\sum_{k>n}(k-n)\gamma_k(u_0)&\geq & \frac 12 \sum_{k>2n}k \gamma_k(u_0) \\
&\gtrsim & \frac{1}{n^{2s}[\log (n)]^2} \, ,
\end{eqnarray*}
it then follows from \eqref{estimate for w_0} that
$$
n^{s+\e} ( \log (n) )^{-2}|\hat w_0(n)|=\ell^2_n
$$
and hence $w_0\in H_+^{s+\delta }$ for any $\delta <\e $. On the other hand, by the definition of $w_0$,
\begin{eqnarray*}
i\Pi [e^{i \partial_x^{-1}u_0}w_0]&=&\Pi [e^{i\partial_x^{-1}u_0}\Pi (u_0e^{-i\partial_x^{-1}u_0})]\\
&=&\Pi u_0-\Pi [e^{i\partial_x^{-1}u_0}(\text{Id}-\Pi) (u_0e^{-i\partial_x^{-1}u_0})] \, .
\end{eqnarray*}
By Lemma \ref{Hankel}(iii) (with $\alpha = 1$, $\beta = s + 1/2$)
 it then follows that
$ i \Pi [e^{i \partial_x^{-1}u_0} w_0] = \Pi u_0 + H_+^{2s+1/2}$. 
Since $e^{i\partial_x^{-1}u_0}\in H_c^{1+s}$, $w_0\in H_+^{s+\delta }$, and hence
$ i \Pi [e^{i \partial_x^{-1}u_0} w_0]  \in H^{s+\delta}_+$
one concludes that $u_0\in H^{s+\delta}_{r,0}$.
Hence for any $\delta <\e $, 
$$
\sum_{n=1}^\infty n^{1+2(s+\delta) }\gamma_n(u_0)<\infty \, ,
$$
which is in contradiction to \eqref{gammacritique}. 
Therefore \eqref{to be contradicted} is false and hence so is \eqref{assumption 1}.

To finish the proof of Theorem \ref{smoothing 1}, assume that $r(t)$ is $H_r^{3s+\e }$.
Since by \eqref{smoothing 1 of u by w_L B} and \eqref{smoothing 2 of u by w_L B},
$$
r(t)-r_{\ast}(t)=2{\rm Re}[e^{i\partial_x^{-1}u(t)}(w_L(t)-w_{L,\ast}(t))] \,  ,
$$
estimate \eqref{smoothing 2 of u by w_L B} implies that for $\e >0$ sufficiently small 
 $$
 2{\rm Re} [e^{i\partial_x^{-1}u(t)}(w_L(t)-w_{L,\ast}(t)) ]\in H^{3s+\e} \, .
 $$
 Note that $ \Pi \big(  2{\rm Re} [e^{i\partial_x^{-1}u(t)}(w_L(t)-w_{L,\ast}(t)) ] \big) \in  H^{3s+\e}$ equals
 $$
 \Pi[e^{i\partial_x^{-1}u(t)} \big(w_L(t)-w_{L,\ast}(t) \big) ] 
 + \Pi[e^{-i\partial_x^{-1}u(t)} \big(\overline{w_L(t)}- \overline{w_{L,\ast}(t)} \big) ] \, .
 $$
 By Lemma \ref{Hankel}(iii) (with $\alpha = 1$ and $\beta= s +1/2$), one then concludes that
 for $\e$ sufficiently small
 $$
 T_{e^{i\partial_x^{-1}u(t)}} [w_L(t)-w_{L,\ast }(t)  ]
 = \Pi [e^{i\partial_x^{-1}u(t)}(w_L(t)-w_{L,\ast }(t) ) ] \in H_+^{3s+\e} \, .
 $$
 Since the Toeplitz operator $T_{e^{i\partial_x^{-1}u(t)}} :  H_+^{3s+\e} \to  H_+^{3s+\e}$ is a linear isomorphism 
%
 one obtains
 $$ w_L(t)-w_{L,\ast }(t)\in H^{3s+\e} \, ,
 $$
 which contradicts the above conclusion that  \eqref{to be contradicted} is false. 
 
 This finishes the proof of Theorem \ref{smoothing 1}.
\end{proof}
{ \begin{proof}[Proof of Corollary \ref{smoothing of u by w_L}]
The claimed results can be proved in a similar way as 
Theorem  \ref{smoothing 1} and hence we leave the details of the proof to the reader.
\end{proof}


\section{Benjamin--Ono flow and H\"older spaces}\label{BO on Hoelder}

As an illustration of possible applications of Theorem \ref{smoothing 1}, we show 
in this section how this theorem can be used to study the action of the Benjamin--Ono flow $\mathcal S(t)$ on H\"older 
spaces $C^{\alpha}(\T, \R)$, $1/2 < \alpha < 1$.
The main result of this section is Proposition \ref{BO on C^alpha} below, which states that
for almost any time $t \in \R$ and any $1/2 < \alpha < 1$, $\mathcal S(t)$ does not map 
$C^{\alpha -}(\T, \R)$ into $\cup_{\e >0}C^{\alpha -1/2+\e} (\T, \R)$, whereas
by the Sobolev embedding theorem, $\mathcal S(t)$ maps 
$C^{\alpha -}(\T, \R)$ into $C^{(\alpha -1/2)-} (\T, \R)$.
Here for any $0 < \beta < 1$,
$$
C^{\beta-}(\T, \C):= \bigcap_{\gamma < \beta} C^\gamma (\T, \C) \, .
$$
\indent
First we need make some preliminary considerations.
We begin by reviewing a result on the flow map of the linear Schr\"odinger equation on $\T$,
$$
- i \partial_t \psi = \partial^2_x  \psi \, ,
$$
related to the Talbot effect -- see \cite[Section 2.3 and references therein]{ET3}.
First we need to introduce some additional notation.
For any $0 <  \alpha < 1$, we denote by $C^\alpha (\T, \C)$ the Banach space of $\alpha$-H\"older continuous 
functions $\psi: \T \to \C$, endowed with the standard norm,
$$
\|\psi\|_{C^\alpha} := \sup_{x \in \T} |\psi(x)| + \sup_{x\ne y}\frac{| \psi(y) - \psi(x) |}{ (d(x,y)) ^\alpha} \, , 
$$
where  $ d(x,y)$ denotes the distance between $x$ and $y$ in $\T$.
\begin{theorem}\label{smoothingschrod} (\cite[Theorem 2.16]{ET3})
There exists a subset $N\subset \R $ of Lebesgue measure $0$ so that for any $t\notin N$
and any function $\psi : \T \to \C$ of bounded variation,
${\rm e}^{it\partial_x^2} \psi \in C^{1/2-}(\T, \C)$.
\end{theorem}
\begin{remark}
It follows from the proof of Theorem 2.16 in \cite{ET3} (cf. p. 37--39) that 
 the set $N$ of Theorem \ref{smoothingschrod} can be chosen independently of $\psi$.
 \end{remark}
Theorem \ref{smoothingschrod} can be used to analyze the action of the flow map  ${\rm e}^{it\partial_x^2}$
on $C^{1/2-}(\T, \C)$. To state our result we make the following preliminary considerations.
The $\C$-vector space $C^{1/2-}(\T, \C)$ is endowed with the countable family of the norms of $C^{1/2-1/p} (\T, \C)$, $p \in \Z_{\ge 3}$.
In this way, $C^{1/2-}(\T, \C)$ becomes a Fr\'echet space with the property that $C^\infty (\T, \C)$ is dense in $C^{1/2-}(\T, \C)$, 
since $C^\infty (\T, \C)$ is dense in $C^\alpha (\T, \C)$, when considered with the norm of $C^{\alpha -\e}(\T, \C)$, $\e > 0$.
\begin{corollary}\label{schr on hoelder}
Let $N$ be the set of Lebesgue measure zero of Theorem \ref{smoothingschrod}.
Then for any $t\notin N':= - N $, 
${\rm e}^{it\partial_x^2}$ does not map $C^{1/2-}(\T, \C)$ into $L^\infty (\T, \C)$.
\end{corollary}
\begin{proof} 
Let $t\notin N'$ and 
suppose that for any $\psi \in C^{1/2-}(\T, \C)$, ${\rm e}^{it\partial_x^2}\psi \in L^\infty (\T, \C)$. 
Since for any $n \in \Z$,
$\widehat{{\rm e}^{it\partial_x^2} \psi}(n) = {\rm e}^{-it n^2}  \widehat{\psi}(n)$,
it then easily follows from the closed graph theorem that the linear map 
$$
{\rm e}^{it\partial_x^2} : C^{1/2-}(\T, \C)\to L^\infty (\T, \C)
$$
is continuous. 
In particular, since $C^\infty (\T, \C)$ is dense in $C^{1/2-}(\T, \C)$ by the considerations above, 
one concludes that the image of $C^{1/2-}(\T, \C)$ by ${\rm e}^{it\partial_x^2}$ is contained 
in the closure of $C^\infty (\T, \C)$ in $L^\infty (\T, \C)$. 
As a consequence, the image of $C^{1/2-}(\T, \C)$ by ${\rm e}^{it\partial_x^2}$ consists of functions 
which are almost everywhere equal  to a continuous function. 

Consider a function $\psi$  of bounded variation, which is not continuous, e.g.,  a step function. 
Then $\psi$ is not almost everywhere equal to a continuous function. 
On the other hand, since $-t \notin N$, it follows from Theorem \ref{smoothingschrod} that
$$
\phi :={\rm e}^{-it\partial_x^2} \psi \in C^{1/2-}(\T, \C) \, ,
$$
and hence in contradiction to our choice of $\psi$, ${\rm e}^{it\partial_x^2} \phi = \psi$ would have to be 
almost everywhere equal to a continuous function.
\end{proof}
Let us now turn to the flow map $\mathcal S(t)$ of the Benjamin-Ono equation.
In view of Corollary \ref{schr on hoelder}, we begin by studying the action of $\mathcal S(t)$
on $C^{1/2-}(\T, \C)$.
To this end we need to establish the following auxilary result on the action of $\mathcal S(t)$ 
on the Besov space $B^1_{1,1}(\T, \R)$. Recall that for any $s \ge 0$ and $p \ge 1$, $B^{s}_{p,1}(\T, \mathbb K)$, $\mathbb K \in \{ \R, \C\}$,
is a Banach space, endowed with the norm
\begin{equation}\label{Besov space norm}
\| f\|_{B^s_{p,1}} :=  \sum_{j \ge 0} 2^{s j} \|P_j f \|_{L^p} ,
\end{equation}
where $P_j$, $j \ge 0$, are Littlewood-Paley projections (cf. e.g. \cite[Apppendix 2.6]{P}).
Note that elements in $B^1_{1,1}(\T, \R)$ are absolutely continuous and hence of bounded variation.
Furthermore, $B^1_{1,1}(\T, \R)$ is a Banach algebra.
\begin{lemma} \label{smoothingBO} 
Let $N$ be the set of Lebesgue measure zero of Theorem \ref{smoothingschrod}.
Then for any $t \notin N$ and any $u_0\in B^1_{1,1}(\T, \R)$ with $\la u_0 \, | \,1\ra = 0$, 
$\mathcal S(t,u_0)\in C^{1/2-}(\T, \R)$.
\end{lemma} 
\begin{proof} For any given  $u_0\in B^1_{1,1}(\T, \R)$ with $\la u_0 \, | \,1\ra = 0$, let
$w_0:= -i\Pi (u_0{\rm e}^{-i\partial_x^{-1}u_0})$.
Note that the Szeg\H{o} projection $\Pi$ maps $B^1_{1,1}(\T, \C)$ into itself (cf. e.g. \cite[Apppendix 2.6]{P})
and hence  $w_0\in B^1_{1,1}(\T, \C)$. 
Since any function in $B^1_{1,1}(\T, \C)$ is of bounded variation, 
it then follows by Theorem \ref{smoothingschrod} that for any $t \notin N$, 
$$
w_L(t):={\rm e}^{ it(\partial_x^2+\la u_0^2 | 1 \ra)} w_0 \in C^{1/2-}(\T, \C) \, .
$$
Furthermore, $B^1_{1,1}(\T, \R)\subset B^{1/2}_{2,1}(\T, \R)\subset H^{1/2}(\T, \R)$, and therefore
$$
u(t):=\mathcal S(t,u_0)\in H^{1/2}_r \, ,\qquad 
{\rm e}^{i\partial_x^{-1}u(t)}\in H^{3/2}_r \subset \bigcap _{\alpha <1}C^\alpha (\T, \R) \, .
$$
As a consequence, by Theorem \ref{smoothing 1},
$$
u(t)=2{\rm Re}\big({\rm e}^{i\partial_x^{-1}u(t)}iw_L(t) \big)+r(t) \, ,
$$
where 
$$
r(t)\in H^{3/2 -}_r\subset \bigcap _{\alpha <1}C^\alpha (\T, \R) \, .
$$
Altogether we thus proved that $u(t)\in C^{1/2 -}(\T, \R)$. 
\end{proof}
As already advertised above,
the following proposition states a result on the action of the  BO flow map on $C^{1/2-}(\T, \R)$.
It should be compared with the result of Corollary \ref{schr on hoelder} on the action of 
the flow map ${\rm e}^{it\partial_x^2}$ of the linear Schr\"odinger equation on $C^{1/2-}(\T, \C)$.
\begin{proposition}\label{BO on C^{1/2}}
Let $N$ be the set of Lebesgue measure zero of Theorem \ref{smoothingschrod}.
Then for any $t\notin N':= -N$,
 $\mathcal S(t)$ does not map $C^{1/2-}(\T, \R)$ into $\cup_{\e >0}C^\e (\T, \R)$. 
\end{proposition}
\begin{proof}
Given any $u_0\in B^1_{1,1}(\T, \R)$ with  $\la u_0 \, | \,1\ra = 0$, it follows from Lemma \ref{smoothingBO}
that for any  $t \notin  N'$ (and hence $-t \notin N$), $v:= \mathcal S(-t, u_0) \in C^{1/2-}(\T, \R)$.
Since $\mathcal S(t, v)=u_0$, the proposition is proved by choosing 
$$ 
u_0\in B^1_{1,1}(\T, \R)\setminus \bigcup_{\e >0}C^\e (\T, \R) \ .
$$
A possible choice is
$$
u_0(x)={\rm Re} \big(\sum_{j=0}^\infty \frac{1}{j^2}{\rm e}^{i2^j x}\chi_j(x) \big) \, ,  \qquad
 \chi_j(x):=\sum_{k\in \Z}\chi (2^j (x-2k\pi)) \, ,
 $$
where $\chi: \R \to \R$ is $C^\infty$-smooth with $\chi (0)=1$ and support contained in the open interval  $(-1/4,1/4)$.
\end{proof}
Let us now turn to the main result of this section, which concerns the action of the BO flow map on $C^{\alpha -}(\T, \R)$ with $1/2 < \alpha < 1$.
To prove this, we first extend  Theorem \ref{smoothingschrod} and Lemma \ref{smoothingBO}.
The extension of Theorem \ref{smoothingschrod} is obtained from the proof of the latter in a straightforward way
and reads as follows.
\begin{corollary}
Let $N$ be the set of Lebesgue measure zero of Theorem \ref{smoothingschrod}.
Then for any $t\notin N$ and any $1/2 < \alpha < 1$, 
$$
{\rm e}^{it\partial_x^2}: B^{\alpha +1/2}_{1,1}(\T, \C) \to C^{\alpha-}(\T, \C)  \, .
$$
\end{corollary}
\begin{proof}
Let $\beta := \alpha -1/2$. Given $\psi \in B^{\alpha +1/2}_{1,1}(\T, \C)$, we write $\psi$ as
 $$
 \psi =\la \psi \,  | \, 1 \ra + |D|^{-\beta} \phi \, ,
 $$
 where $\phi := |D|^\beta \psi \in B^1_{1,1}(\T, \C)$. Then 
 $$
 {\rm e}^{it\partial_x^2} \psi = \la \psi \,  | \, 1 \ra + |D|^{-\beta} {\rm e}^{it\partial_x^2}\phi \in C^{\alpha-}(\T, \C)  \, .
 $$
 \end{proof}
 By the arguments used in its proof,  Lemma \ref{smoothingBO} can be extended as follows.
\begin{lemma}\label{smoothingBO alpha}
Let $N$ be the set of Lebesgue measure zero of Theorem \ref{smoothingschrod} and $1/2 < \alpha < 1$.
Then for any $t \notin N$ and any $u_0\in B^{\alpha +1/2}_{1,1}(\T, \R)$ with $\la u_0 \, | \,1\ra = 0$, 
$S(t,u_0)\in C^{\alpha-}(\T, \R)$.
\end{lemma}
Now we are ready to state our result on the action of $\mathcal S(t)$ on $C^{\alpha -}(\T, \R)$.
\begin{proposition}\label{BO on C^alpha}
For any $1/2 < \alpha < 1$ the following holds.\\
(i)  For any $t\in \R$,
$\mathcal S(t)$ maps $C^{\alpha -}(\T, \R)$ continuously  into $C^{(\alpha -1/2)-}(\T, \R)$.\\
(ii) Let $N$ be the set of Lebesgue measure zero of Theorem \ref{smoothingschrod}.
Then for any $t \notin N':= - N$,
 $\mathcal S(t)$ does not map $C^{\alpha -}(\T, \R)$ into $\cup_{\e >0}C^{\alpha -1/2+\e} (\T, \R)$.
\end{proposition}
\begin{proof}
(i) Note that for any $0 < \beta < 1$, $C^\beta (\T, \R) \subset H^{\beta -}_r:= \cap_{\e > 0}H^{\beta - \e}_r$. 
Furthermore, if $\beta >1/2$, then $H^{\beta }_r \subset C^{\beta -1/2}(\T, \R )$ by the Sobolev embedding theorem. 
Now let $u_0\in C^{\alpha-}(\T, \R)$ with $1/2 < \alpha < 1$. Then $u_0\in H^{\alpha -}_r$ 
and hence $\mathcal S(t,u_0)\in H^{\alpha-}_r \subset C^{(\alpha-1/2)-}(\T, \R)$ for any $t \in \R$. \\
(ii)
We argue as in the proof of Proposition \ref{BO on C^{1/2}}.
Given any $u_0\in B^{\alpha + 1/2}_{1,1}(\T, \R)$ with  $\la u_0 \, | \,1\ra = 0$, it follows from Lemma \ref{smoothingBO alpha}
that for any  $t \notin  N'$ (and hence $-t \notin N$), $v:= \mathcal S(-t, u_0) \in C^{\alpha-}(\T, \R)$.
Since $\mathcal S(t, v)=u_0$, the proposition is proved by choosing 
$$ 
u_0\in B^{\alpha+ 1/2}_{1,1}(\T, \R)\setminus \bigcup_{\e >0}C^{ \beta + \e} (\T, \R) \ , \qquad \beta:= \alpha - 1/2 \, .
$$
A possible choice is
$$u_0(x)={\rm Re} \big(\sum_{j=0}^\infty \frac{2^{-j \beta}}{j^2}{\rm e}^{i2^j x}\chi_j(x) \big) \, , \qquad 
\chi_j(x):=\sum_{k\in \Z}\chi (2^j (x-2k\pi)) \, ,
$$
where $\chi: \R \to \R$ is $C^\infty$-smooth with $\chi (0)=1$ and support contained in the open interval  $(-1/4,1/4)$.
\end{proof}


\medskip

 \appendix
 
 \section{Smoothing properties of Hankel operators}\label{Hankel operators}
 
In this appendix, we record results on smoothing properties of Hankel operators, which are used throughout the paper.
First we need to introduce some  more notation.
For any $s \in \R,$ $H^s_-$ denotes the Hardy space with Sobolev exponent $s$, i.e.,
$$
H^s_- := 
\{ f \in H^s_c \ : \ \widehat f(n) = 0 \, \, \, \forall n >0  \} \, 
$$
and $\Pi^-$ the corresponding projection, 
$$
\Pi^- : H^s_c \to H^s_- \, , \ f = \sum_{n \in \Z} \widehat f(n) e^{inx}  \mapsto f= \sum_{n \le 0} \widehat f(n) e^{inx} \, .
$$
For any $u \in H^{1}_c$, denote by $H_u : H^0_- \to H^0_+$ 
and $H^-_u : H^0_+ \to H^0_-$ 
the Hankel operators with symbol $u$, defined as
$$
H_u f := \Pi[u f], \quad \forall \, f \in H^0_- \, , \qquad 
H^-_u f := \Pi^-[u f], \quad \forall \, f \in H^0_+ \, .
$$ 
Actually, $H_u$ and $H_u^-$ extend as bounded linear operators, 
$$
H_u: H^s_- \to H^s_+ \, , \qquad H^-_u: H^s_+ \to H^s_- \, ,
$$ 
for any $-1/2 < s < 0$
(see e.g.  \cite[Lemma 1]{GKT2}).
The following lemma shows that the operators $H_u$ and $H_u^-$  can be defined for symbols $u$ in $H^{1/2}_c$
on appropriate Hardy spaces and that they have smoothing properties, which depend on $s$.
For notational convenience, we itemize them according to the size of the gain of regularity.
For $s < 1/2$ and $\alpha \ge 0$, let
\begin{equation}\label{def beta}
\beta \equiv \beta(s, \alpha) := \alpha + s  -  \frac 12   < \alpha \, .
\end{equation}
\begin{lemma}\label{Hankel} (Smoothing properties of Hankel operators.)
For any $u \in H^{s+\alpha}_c$ and $f \in H^s_-$ with $s \in \R$, 
$\alpha \ge 0$,  the following holds:
\begin{itemize}
\item[(i)]  $\| H_u [f] \|_{s + \alpha} \lesssim_{s, \alpha} \|u\|_{s+\alpha} \, \| f \|_s$, \qquad  \mbox{if} $ s > \frac 12,$ $\alpha \ge 0$.
\item[(ii)] $\| H_u [f] \|_{\frac 12 + \alpha - \varepsilon} \lesssim_{ \alpha, \varepsilon} \|u\|_{\frac 12 +\alpha} \, \| f \|_{\frac 12}$,
\  \ \mbox{if} $ s=\frac 12$, $ \alpha \ge 0$, $ \varepsilon > 0$.
\item[(iii)]  $\| H_u [f] \|_{s + \beta } \lesssim_{s, \alpha} \|u\|_{s+\alpha} \, \| f \|_s$,  \qquad  \mbox{if} $ 0 \le s < \frac 12$, $\alpha \ge \frac 12 - s$.
\item[(iv)] $\| H_u [f] \|_{s + \beta} \lesssim_{s, \alpha} \|u\|_{s+\alpha} \, \| f \|_s$,  \qquad \mbox{if} $ s < 0$, $\alpha \geq \frac 12 - s$, $\alpha > - 2 s ,$ 
\end{itemize}
where $\beta \equiv \beta(s, \alpha)$ is given by \eqref{def beta}.
Corresponding results hold for the operator $H_u^-$.
\end{lemma}
\begin{proof}
Let $u = \sum_{k \in \Z} \widehat u(k) e^{ikx} \in H^{s+\alpha}_c$ and $f = \sum_{p \ge 0} \widehat f(-p) e^{-ipx} \in H^s_-$. 
Then with $n := - (k+p)$,
$$
g(x) := \Pi \big( \sum_{k \in \Z, p \ge 0}  \widehat u(-k) \widehat f(-p) e^{-i(k+p)x}  \big) = \sum_{n \ge 0} \widehat g(n) e^{inx} 
$$
where
$$
 \widehat g(n) := \sum_{p \ge 0} \widehat u(n+p) \widehat f(-p) \, , \qquad \forall \, n \ge 0 \, .
$$
By the Cauchy Schwarz inequality, one obtains
$$
 | \widehat g(n) |^2 \le \|f \|_s^2 \sum_{p \ge 0} \frac{1}{\langle p \rangle^{2s} }|\widehat u(n+p) |^2 \, , \qquad \forall \, n \ge 0\, ,
$$
and thus for any $\gamma \in \R$,
\begin{equation}\label{estimate for s+gamma norm}
\| g \|^2_{s+\gamma}  
 \le  \|f \|_s^2 \sum_{\ell \ge 0} |\widehat u(\ell) |^2   \sum_{p,n \ge 0,p + n = \ell} \frac{\langle n \rangle^{2(s+\gamma)}}{\langle p \rangle^{2s} } \, .
\end{equation}
(i) In the case $s > 1/2$, $\alpha \ge 0$ one has $2s > 1$ and hence 
$$
\sum_{p,n \ge 0,p + n = \ell} \frac{\langle n \rangle^{2(s+\alpha)}}{\langle p \rangle^{2s} } 
\le \langle \ell \rangle^{2(s+\alpha)} \sum_{0 \le p \le \ell} \frac{1}{\langle p \rangle^{2s}}
\lesssim_s \langle \ell \rangle^{2(s+\alpha)}  ,
$$
so that by \eqref{estimate for s+gamma norm},
$$
\| g \|_{s+\alpha} \lesssim_s \| u \|_{s+\alpha} \|f \|_s  .
$$
(ii) Recall that $s = 1/2$, $\alpha \ge 0$ and note that without loss of generality, we can assume that $0 < \varepsilon < 1/2$. Then
$$
\sum_{p,n \ge 0, p + n = \ell} \frac{\langle n \rangle^{2(s+ \alpha - \varepsilon)}}{\langle p \rangle^{2s} } 
\le \langle \ell \rangle^{2(\frac 12+\alpha - \varepsilon)} \sum_{0 \le p \le \ell} \frac{1}{\langle p \rangle}
\lesssim  \langle \ell \rangle^{2(\frac 12+\alpha - \varepsilon)} \log \langle \ell \rangle .
$$
Hence \eqref{estimate for s+gamma norm} implies that
$$
\| g \|_{\frac 12+\alpha - \e} \lesssim_{s, \varepsilon}  \| u \|_{\frac 12+\alpha}  \|f \|_\frac 12  \, .
$$
(iii) In the case $0 \le  s < 1/2$, $\alpha \ge 1/2 - s$,  one has $1-2s > 0$ and 
$s + \beta = s + \alpha -(1/2 - s) \ge 0$, implying that
$$
\sum_{p,n \ge 0,p + n = \ell} \frac{\langle n \rangle^{2(s+\beta)}}{\langle p \rangle^{2s} } 
\le \langle \ell \rangle^{2(s+\beta)} \sum_{0 \le p \le \ell} \frac{1}{\langle p \rangle^{2s}}
\lesssim_s  \langle \ell \rangle^{2(s+\beta)} \langle \ell \rangle^{1 - 2s} .
$$
Since $s + \beta + 1/2  -s = s+  \alpha$, it then follows by \eqref{estimate for s+gamma norm} that
$$
\| g \|_{s+ \beta} \lesssim_s \| u \|_{s+\alpha} \|f \|_s  \, .
$$
(iv) In the case $ s < 0$ and $\alpha \geq \frac 12+|s|$, $\alpha >2|s|$,  one has 
$$
s+ \beta =  - \frac 12 + \alpha -2|s| > - 1/2
$$ 
and hence
$$
\sum_{p,n \ge 0,p + n = \ell} \frac{\langle n \rangle^{2(s+\beta)}}{\langle p \rangle^{2s} } 
\le  \langle \ell \rangle^{- 2s}  \sum_{0 \le n \le \ell}  \langle n  \rangle^{2(s+\beta)}
\lesssim_s  \langle \ell \rangle^{-2s} \langle \ell \rangle^{1 + 2s+2 \beta} \,  .
$$
Since $1/2 + \beta  =  s + \alpha $, it then follows by \eqref{estimate for s+gamma norm} that
$$
\| g \|_{s+ \beta} \lesssim_s \|f \|_s  \, \| u \|_{s+\alpha} \, .
$$
\end{proof}


 \section{On diffeomorphism properties of Tao's gauge transform}\label{Appendix C}
 
 The aim of this appendix is to prove diffeomorphism properties of Tao's gauge transform. 
 Without further reference, we will use the notation introduced in the main body of the paper.
 \begin{theorem}\label{diffeogauge} For any $s \ge 0$, Tao's gauge transform
$$
\mathcal G : H^s_{r,0} \to H^s_{+,0} ,  \, u\mapsto \partial_x\Pi (e^{-i\partial_x^{-1}u}) \, ,
$$
is a real analytic diffeomorphism onto an open subset of $H^s_{+,0}$, which is proper.
In particular, the functions $in e^{inx}$, $n \ge 1$, do not belong to the range of $\mathcal G$
and the differential of $\mathcal G$ at $u=0$ reads $d_0 \mathcal G   = - i \Pi$.
\end{theorem}
Before proving  Theorem \ref{diffeogauge}, we make some preliminary considerations.
For any given any  $u\in L^2_{r,0}$,
let $w=\mathcal G(u)$ and introduce
\begin{equation}\label{def v plus}
v:=e^{-i\partial_x^{-1}\Pi u} \in H^1_+ \, .
\end{equation}
Since $u=\Pi u+\overline{\Pi u}$,  one has
\begin{equation}\label{identity for u}
e^{-i\partial_x^{-1} u} = e^{-i\partial_x^{-1}\Pi u} \cdot e^{- i\partial_x^{-1}\overline{\Pi u}} = \frac {v}{\overline{v}}
\end{equation}
and consequently,
\begin{equation}\label{identity w}
\partial_x^{-1}w = \partial_x^{-1} \partial_x \Pi [e^{-i\partial_x^{-1} u} ] = \Pi [\,  \frac {v}{\overline{v}} \, ] + a \,  ,  \quad
a:=\textcolor{red}{-} \la e^{-i\partial_x^{-1} u}  | 1 \ra \, .
\end{equation}
Given $g \in H^1_c$, we denote by $\check H_{ g}$ the anti-linear Hankel operator of symbol $g$, 
$$
\check H_g : H_+ \to H_+ , \,  h \mapsto \Pi [g \overline h] \, .
$$
\begin{lemma}\label{kernel}
For any $w \in \mathcal G(L^2_{r,0})$, 
the nullspace $\ker ({\rm Id}- \check H_{\partial_x^{-1}w})$ of the linear operator ${\rm Id}- \check H_{\partial_x^{-1}w} : H_+ \to H_+$ satisfies
$$
\ker ({\rm Id}- \check H_{\partial_x^{-1}w})\cap H_{+,0} =\{ 0\} \, .
$$
\end{lemma}
\begin{proof}[Proof of Lemma \ref{kernel}]
Let $h$ be an element in $\ker ({\rm Id}- \check H_{\partial_x^{-1}w})\cap H_{+,0}$.
Then $h=\Pi [(\partial_x^{-1}w) \overline h ]$ and by \eqref{identity w}
$$
\Pi [ \, \frac{v}{\overline{ v}} \,  \overline h \, ] = \Pi [ (\Pi [ \frac{v}{\overline{v}}] ) \, \overline h  ]
=\Pi [(\partial_x^{-1} \, w) \,  \overline h ] - a \, \Pi [ \, \overline h \,]  \, .
$$
Since $h \in H_{+,0}$, one has $\Pi [ \, \overline h \,] = \la 1 |  h \ra = 0$.
Using that  $h=\Pi [(\partial_x^{-1}w) \overline h ]$ it then follows that
$$
\Pi [ \, \frac{v}{\overline{ v}} \,  \overline h \, ] = h \, ( = \Pi h) \,  .
$$
Hence there exists $f\in H_{+,0}$ so that
\begin{equation}\label{eq1:h}
 \frac{v}{\overline{ v}} \, \overline h = h +  \overline f \, .
\end{equation}
This implies that
\begin{equation}\label{identity for v_+ f}
\overline{v} \, \overline f = v \overline h -  \overline{v} h 
= \overline {(\overline{v} h -  v \, \overline h  )   } =  - v f  \in H_+ \, ,
\end{equation}
where we used that by \eqref{def v plus}, $v$ is in $H^1_+$.
Consequently\footnote{Throughout this appendix, we make frequent use of the elementary observation that any function $f\in H_+$ with $\overline f \in H_+$, is constant.}, 
$v f$ is a constant function. 
Furthermore, since $v$ and $f$ both belong to $H_+$ and  $\la f |  1 \ra =0$,
$$
\la v f |  1\ra =\la v |  1\ra \cdot  \la f \, | \, 1\ra = 0 \, .
$$
Therefore  $v f=0$ and  in turn, since $v$ does never vanish,   $f=0$. Coming back to \eqref{eq1:h}, we conclude that
$$
 \frac{\overline h}{\overline{ v}} = \frac{h}{v}  \in H_+ \,  .
$$
We thus again conclude that the function $h/v$ is constant. Since 
$$
\la \frac{h}{v} | \, 1\ra = \la h | \, 1\ra \cdot \la \frac{ 1}{v} | \, 1\ra = 0
$$
we infer that $h=0$.
\end{proof}
\begin{proof}[Proof of Theorem \ref{diffeogauge}]
In a first step we consider the case where $s=0$. It is straightforward to verify that
$\mathcal G : L^2_{r,0} \to H_{+,0}$ is real analytic and  $d_0 \mathcal G   = - i \Pi$.
Next we prove that for any integer $n \ge 1$, the function
$f_n(x) := ine^{inx}$
is not an element in the image $\mathcal G(L^2_{r,0})$ of $\mathcal G$.
In the case where $n \ge 2$ we argue as follows. Note that for $n \ge 2$,
$$ 
h_n (x) :=e^{ix}+e^{i(n-1)x} \in H_{+, 0}  
$$
and since $\partial_x^{-1}f_n =e^{inx}$, one has
$h_n - \check H_{\partial_x^{-1}f_n}[h_n] = 0 $
and hence
$$ 
h_n \in \ker ({\rm Id}- \check H_{\partial_x^{-1}f_n}) \cap  H_{+, 0} \, .
$$
By Lemma \ref{kernel} one then concludes that  $f_n\notin \mathcal G(L^2_{r,0})\ .$ 
In the case $n =1$ we have to argue differently since $h_1 := e^{ix}+ 1$ is not an element in $H_{+,0}$.
We note that $\mathcal G$ possesses the following scaling invariance : for any $u \in L^2_{r,0}$ and any integer $n\geq 1$,
$$
\mathcal G(u_n) (x) =n\mathcal G(u) (nx)\, , \qquad u_n(x) := n u(nx) \, .
$$
Consequently, if $f_1(x)= e^{ix}$ were to belong to the range of $\mathcal G$, then so would $f_n(x) = ne^{inx} = n f_1(nx)$ for any $n \ge 2$,
in contradiction to what we just have proved.

Next we establish that $\mathcal G: L^2_{r, 0} \to H_{+,0}$ is injective. Assume that $u_1$, $u_2 \in L^2_{r,0}$ satisfy
$\mathcal G(u_1)=\mathcal G(u_2)$.
Set
$$
v_1:=e^{-i\partial_x^{-1}\Pi u_1}\in H^1_+ \,  , \qquad  v_2:=e^{-i\partial_x^{-1}\Pi u_2}\in H^1_+ \,  .
$$
By \eqref{identity for u}, the assumption $\mathcal G(u_1)=\mathcal G(u_2)$ can then be written as
\begin{equation}\label{identity for u_1,u_2}
\partial_x\Pi \big[ \, \frac{v_1}{\overline{v_1}}  -\frac{v_2}{\overline{v_2}}  \, \big]=0 \, .
\end{equation}
It means that there exists $f \in H_+$ so that
\begin{equation}\label{eq2:ab}
\frac{v_1}{\overline{v_1}}  -\frac{v_2}{\overline{v_2}} \,  = \, \overline f \, .
\end{equation}
Arguing as in \eqref{identity for v_+ f} it then follows that
$$
\overline{v_1} \,  \overline{v_2} \, \overline f
=v_1 \, \overline{v_2} - v_2 \, \overline{v_1} = - \overline {( v_1 \, \overline{v_2} - v_2 \, \overline {v_1} )}= - v_1 v_2 \, f \, ,
$$
implying that $v_1 v_2 \, f $ is a constant function, $v_1 v_2 \, f  = a \in \C$. Coming back to \eqref{eq2:ab}, we obtain
\begin{equation}\label{eq3:ab}
\frac{v_1}{v_2}-\frac{\overline{v_1}}{\overline{v_2}} = \frac{\overline a}{v_2 \, \overline{v_2}} \,  .
\end{equation}
Note that 
\begin{equation}\label{mean of v_1}
v_1 =1+ r\, , \qquad r:=  \sum_{k=1}^\infty \frac{(-i \partial_x^{-1}\Pi u_1)^k}{k!}  \in H_{+,0} \, .
\end{equation}
Hence $\la v_1 \, |  \, 1 \ra  =1$ and in turn $\la \overline{v_1} \, | \, 1 \ra =1$. 
Substituting $ -u_2$ for $u_1$, \eqref{mean of v_1} yields $\la \frac{1}{v_2} \, | \,  1\ra =1$ and $\la \frac{1}{\overline{v_2}} \, | \,  1 \ra =1$. 
Finally, since $v_1$ and $1/v_2$ both belong to $H_+$, one has
\begin{equation}\label{identity for quotient}
\la \frac{v_1}{v_2} \, | \,  1 \ra  =\la v_1 \, | \,  1\ra \cdot  \la \frac{1}{v_2} \, | \,  1 \ra =1 \,  ,
\end{equation}
and in turn $\la \frac{\overline{v_1}}{\overline{v_2}} \, | \, 1\ra =1$. We then conclude from \eqref{eq3:ab}  that
$$
0 = \la  \frac{\overline a}{v_2 \overline{v_2}}  \, | \, 1\ra = \overline a \, \| \frac{1}{v_2} \| \textcolor{red}{^2} \, ,
$$
implying that $a=0$. By \eqref{eq3:ab} it then follows that
$\frac{v_1}{v_2} = \frac{\overline{v_1}}{\overline{v_2}}$ is a constant, which by \eqref{identity for quotient} equals $1$. 
We thus have proved that $v_1= v_2$ and therefore
$$
 \Pi [u_1] = \frac{1}{v_1} i  \partial_x v_1 = \frac{1}{v_2} i \partial_x v_2 = \Pi[u_2] \,  ,
$$
yielding $u_1= u_2$. This proves the injectivity of $\mathcal G$.

It remains to show that $\mathcal G$ is a local diffemorphism. 
As already pointed out,  $\mathcal G :  L^2_{r,0} \to H_{+,0}$ is a real analytic map. Hence by the inverse function theorem, 
we just need to prove that for any $u\in L^2_{r,0}$, $d_u \mathcal G : L^2_{r,0} \to H_{+,0}$ is a linear isomorphism.

 An easy computation yields that for any $h\in L^2_{r,0}$,
\begin{equation}\label{dG}
d_u\mathcal G[h]= - i\partial_x\Pi [ (\partial_x^{-1}h) e^{-i\partial_x^{-1}u}  \, ] 
= - i\partial_x\Pi [(\partial_x^{-1}h )\frac{v}{\overline v} \, ] \,  ,  
\end{equation}
where $ v:=e^{-i\partial_x^{-1}\Pi u} \in H^1_+$ (cf. \eqref{identity for u}).
First we prove that $d_u\mathcal G$ is one-to-one. Assume that $h$ belongs to the kernel of $d_u \mathcal G$. 
Then $h \in H_{+,0}$ and  there exists $f\in H_+$ so that
$$
(\partial_x^{-1}h) \frac{v}{\overline v} =\overline f \, .
$$
It follows that
$$
\overline v^2 \, \overline f = \partial_x^{-1}h \, v \, \overline v
$$
is real valued and belongs to $H_+$. Hence $\overline v^2 \, \overline f$ is a constant function, $\overline v^2 \, \overline f = a \in \C$, 
implying that
$$
\partial_x^{-1}h = \frac{a}{v \, \overline v} \, .
$$
Taking the inner products of both sides of the latter identity with $1$, we get $a=0$. Since $h \in H_{+,0}$, we conclude that $h=0$,
proving that $d_u \mathcal G$ is one-to-one.

It remains to show that $d_u \mathcal G : L^2_{r,0} \to H_{+,0}$ is onto.
Since $d_u \mathcal G$ is one-to-one, it suffices to prove that for any $u \in L^2_{r,0}$, $d_u\mathcal G$ is a compact perturbation of a linear isomorphism. 
For any $h\in L^2_{r,0}$, one has
$$
h= h_1 + h_2 \, , \qquad  h_1 =\Pi [h] \in H_{+,0}  \, , \quad h_2 = \overline{h_1} = (\mbox{Id} - \Pi ) [h]  \, ,
$$
and hence by \eqref{dG}, 
$$
d_u \mathcal G[h] = L_1[h_1]+L_2[h_2]+L_3[h_1]+L_4[h_2] \, , \qquad \forall h \in L^2_{r,0} \, ,
$$
where $L_1, L_3 : H_+ \to H_+$ are the bounded linear operators, 
$$
L_1[g] := -i \Pi [ \frac{v}{\overline v} \, g ]  \, , \qquad \qquad 
L_3[g] :=-i \Pi [\partial_x \big( \frac{v}{\overline v} \big) \cdot  \partial_x^{-1}g ] \, , \qquad  \, 
$$
and $L_2, L_4 : H_- \to H_+$ the bounded linear operators, 
$$
 L_2[g] := -i \Pi [\frac{v}{\overline v} \, g ] \, , \qquad \qquad 
L_4[g] := -i \Pi [ \partial_x \big (\frac{v}{\overline v} \big) \cdot \partial_x^{-1}  g ] \, . \qquad \ 
$$
By the Sobolev embedding theorem and Rellich's theorem, the bounded linear operator $\partial_x^{-1} : L^2_{r,0} \to H^1_{r,0}$ 
gives rise to a compact linear operator $L^2_{+,0} \to L^\infty_{r,0}$, which we again denote by $\partial_x^{-1}$. 
Hence $L_3$ and $L_4$ are compact operators.
Furthermore, $ \Pi [\frac{v}{\overline v}] \, \in H^1_+$ and $L_2$ is the Hankel operator $H_{-i \Pi [\frac{v}{\overline v}]}$
with symbol $-i \Pi [\frac{v}{\overline v}]$,
$$
L_2[g]=-i  H_{\Pi [\frac{v}{\overline v}]} \, [g] \, , \qquad \forall \, g \in  H_-  \, .
$$
By the smoothing properties of Hankel operators (cf. Lemma \ref{Hankel}(iii) in Appendix \ref{Hankel operators}  with $\alpha = 1, s= 0, \beta= 1/2$) 
it then follows that $L_2: H_- \to H_+$ is compact.
Finally, $L_1: H_+ \to H_+$ is a  Toeplitz operator with symbol $-i \frac{v}{\overline v}$, 
$$
L_1 [g] = -i T_{\frac{v}{\overline v}} [g] \, , \qquad \forall \, g \in H_+ \, ,
$$
which is invertible with inverse given by (cf. e.g. \cite[Lemma 6.5]{GKT4})
$$
L_1^{-1} [f] = i \frac 1 v \, \Pi [\,  \frac{1}{\overline v} \, f ] \,  ,  \qquad  \forall \,  f \in H_+ \,  .
$$
Denote by $\Pi_1$ the projection 
$$
\Pi_1: H_+ \to H_{+, 0} \, ,  \, g \mapsto g - \la g | 1 \ra \, .
$$ 
Since $d_u \mathcal G : L^2_{r,0} \to H_{+,0}$ it follows that for any $h \in L^2_{+,0}$,  $d_u \mathcal G [h]$ equals
$$
  \Pi_1 \circ L_1  [ \Pi h]  +  \Pi_1 \circ L_3  [ \Pi h]  +  
 \Pi_1 \circ L_2 [ (\mbox{Id} - \Pi) h ]  +  \Pi_1 \circ L_4  [(\mbox{Id} - \Pi) h]  
$$
Clearly, the linear operators $\Pi_1 \circ L_3  \circ \Pi : L^2_{r,0} \to H_{+,0}$ and  
$$
\Pi_1 \circ L_2  \circ (\mbox{Id} - \Pi) :  L^2_{r,0} \to H_{+,0} \, , \quad  
\Pi_1 \circ L_4  \circ (\mbox{Id} - \Pi) : L^2_{r,0} \to H_{+,0}
$$  
are compact. Furthermore, one verifies in a straightforward way that
$\Pi_1 \circ L_1  \circ \Pi  : L^2_{+,0} \to H_{+,0}$ is a linear isomorphism (cf. \cite[Lemma 6.5]{GKT4}).
Altogether, we thus have proved that 
$d\mathcal G(u) : L^2_{r,0} \to H_{r,0}$ is a compact perturbation of a linear isomorphism and hence 
a Fredholm operator of index zero. This completes the proof of Theorem \ref{diffeogauge} in the case $s=0$.

By the same arguments as in the above proof one verifies that for any $s > 0$, $\mathcal G : H^s_{r,0}(\T ) \to H^s_{+,0}(\T )$ is a diffeomorphism
onto an open proper subset of  $H^s_{+,0}(\T )$.
\end{proof}
\begin{remark}
The following considerations add to the results on the image of Tao's gauge transform of Theorem \ref{diffeogauge}. 
Consider the family of one gap potentials $u_\alpha \in \cap_{s\ge 0} H^s_{r,0}$, given by 
\begin{equation}\label{one gap}
(\Pi u_\alpha)(x) = \frac{\alpha e^{ix}}{1 - \alpha e^{ix}} =  i \partial_x \log(1 - \alpha e^{ix})
\end{equation}
where $\alpha \in \C$ satisfies $0 < |\alpha | < 1$. Such potentials, studied in \cite[Appendix B]{GK},
give rise to traveling wave solutions of the BO equation. They are one gap potentials in the sense that
$$
\gamma_1 (u_\alpha) = \frac{|\alpha|^2}{1 - |\alpha|^2} \, , \qquad  \gamma_n (u_\alpha) = 0 \, , \quad \forall \, n \ge 2 \, .
$$
Using the identity \eqref{identity for u}, the definition of $\mathcal G(u)$, and the second identity in \eqref{one gap} one sees that 
$$
\begin{aligned}
\mathcal G(u_\alpha) & = \partial_x \Pi [\frac{1 - \alpha e^{ix}} {1 - \overline \alpha e^{-ix}} ]
= \partial_x \Pi[(1 - \alpha e^{ix}) \sum_{k \ge 0} (\overline \alpha e^{-ix})^k] \\
& = \partial_x \big((1 - \alpha e^{ix}) - \alpha e^{ix}\overline \alpha e^{-ix} \big) =  - i \alpha e^{ix} \, .
\end{aligned}
$$
In particular, it follows that any $\beta \in \C$ with $0 \le |\beta| < 1$, $\beta e^{ix}$ is in the image $\mathcal G(L^2_{r,0})$ of $\mathcal G$.

It is then natural to ask whether $\beta e^{ix}$ is in $\mathcal G(L^2_{r,0})$ for some $\beta \in \C$ with $|\beta| \ge 1$.
Using arguments of the proof of Theorem  \ref{diffeogauge}, we now show that 
$\beta e^{ix}$ is {\em not} in $\mathcal G(L^2_{r,0})$ for any $\beta \in \C$ with $|\beta | \ge 1$.
We argue by contradiction and assume that
there exists $u \in L^2_{r,0}$ so that $\mathcal G(u) = \beta e^{ix}$ with $\beta \in \C$ satisfying $|\beta| \ge 1$. 
Following \eqref{def v plus}, define
$v:=e^{-i\partial_x^{-1}\Pi u} \in H^1_+$.
By \eqref{identity for u} one has $e^{-i\partial_x^{-1} u} = \frac {v}{\overline{v}}$, implying that (cf. \eqref{identity w})
$$
\mathcal G(u) = \partial_x \Pi[ \,\frac {v}{\overline{v}} \,] \, .
$$
Applying $\partial_x^{-1}$ to both sides of the latter identity, one concludes that there exists a constant $a \in \C$
so that $- i \beta e^{ix} = \Pi[ \,\frac {v}{\overline{v}}] + a$. It means that there exists $f \in H_+$ so that
$- i \beta e^{ix} - \frac {v}{\overline{v}} = \overline f$ or, multiplying both sides of the latter equation by $\overline v$
\begin{equation}\label{identity beta exp(ix)}
- i \beta e^{ix} \, \overline{v}  -  v = \overline f \, \overline{v} \, .
\end{equation}
Since by \eqref{mean of v_1}, $\la v \, | \, 1 \ra = 1 $  and $\beta e^{ix} ( \overline{v} - 1 ) \in \overline{H_+}$
and hence
$$- i \beta e^{ix} \overline{v} \in  -i \beta e^{ix} +  \overline{H_+} ,
$$
 it then follows from \eqref{identity beta exp(ix)} that
 $ -i \beta e^{ix} - v \in  \overline{H_+}$ and hence is a constant. Using that $\la v \, | \, 1 \ra = 1 $
 we then conclude that
 \begin{equation}\label{identity 0 for v}
  v = 1 - i \beta e^{ix}
  \end{equation} 
  and hence
 \begin{equation}\label{identity 1 for v}
 \partial_x v = \beta e^{ix}  .
 \end{equation}
 On the other hand, by the definition of $v$, one has $v = e^{-i\partial_x^{-1}\Pi u}$ and hence by \eqref{identity 0 for v}
 \begin{equation}\label{identity 2 for v}
   \partial_x v = - i v \Pi u = -i (1 - i \beta e^{ix})  \Pi u.
 \end{equation}
 Combining \eqref{identity 1 for v} and \eqref{identity 2 for v} it then follows that
 $$
 \Pi u =  \frac{ i \beta e^{ix}}{1 - i \beta e^{ix}} .
 $$
 Since by assumption $u \in L^2_{r,0}$, it follows that $|\beta| > 1$. 
 In this case, $\Pi u$ can be written as
 $$
  \Pi u = - \frac{ 1}{1 - \frac{1}{i \beta} e^{-ix}} 
  = - 1 - \sum_{n \ge 1}( \frac{1}{i \beta} e^{ix})^n =  -1 -  \sum_{n \ge 1}\frac{1}{( i \beta)^n} e^{-inx} ,
 $$
 which contradicts that $\Pi u \in  H_{+,0}$.
\end{remark}

 
 \section{High frequency approximations of the differentials $d_u\Phi$ and $d_z \Phi^{-1}$}\label{approximation dPhi}
 
 So far, no high frequency approximation has been found for $\Phi^{-1}$. 
Our goal is to derive such an approximation at least for the differential of $\Phi^{-1}$. 
At the same time we derive a high frequency approximation for the differential of $\Phi$.
  Such approximations are useful for analyzing the pullback of
 vector fields by the maps $\Phi$ and $\Phi^{-1}$.

 Denote by $\mathcal F_{1/2}^+$ the (partial) weighted Fourier transform
 $$
 \mathcal F_{1/2}^+ : H^s_{c} \to \mathfrak h^{s+ \frac 12}  , \, u \mapsto (\frac{1}{\sqrt{n}} \widehat u(n))_{n \ge 1}
 $$
 and by $\mathcal G$ Tao's gauge transform, defined for any given $s \ge 0$ by
 \begin{equation}\label{recall mathcal G}
 \mathcal G : H^s_{r,0} \to H^s_{+,0}  , \, u \mapsto \partial_x \Pi[\overline{g_\infty}] \, , 
 \end{equation}
 where we recall that $g_\infty \equiv  g_\infty (\cdot, u) = e^{i \partial_x^{-1}u}$.
 By \eqref{w and Phi_0}, $\Phi_0$ can be expressed in terms of $\mathcal G$ and $\mathcal F_{1/2}^+$ as
 \begin{equation}\label{formula Phi_0}
 \Phi_0(u) = \frac{1}{i} \mathcal F_{1/2}^+[\mathcal G(u)] \, .
 \end{equation}
  By Theorem \ref{Theorem Phi - Phi_L}, $\frac{1}{i} \mathcal F_{1/2}^+ \circ \mathcal G$ is a high frequency approximation
 of $\Phi$. In more detail,  for any $s \geq 0$,
 $u \mapsto \Phi(u) -  \frac{1}{i} \mathcal F_{1/2}^+[\mathcal G(u)]$ is a continuous map from $H^s_{r,0}$ with values in $\mathfrak h^{s+\frac 12 + \tau(s)}$.
 By \cite{GKT3}, \cite{GKT4},  for any $s \ge 0$, the Birkhoff map $\Phi: H^s_{r,0} \to \frak h^{s+\frac12}$ is a real analytic diffeomorphism and by
 Theorem \ref{diffeogauge} in Appendix \ref{Appendix C},  $\Phi_0:  H^s_{r,0} \to \frak h^{s+\frac12}$  is a real analytic diffeomorphism
 onto an open, proper subset of $\frak h^{s+ \frac12}$}.
 To state our high frequency approximation of the differential of $\Phi$ and of $\Phi^{-1}$ we introduce
$$
\tau_2(s) := \begin{cases}
1 \qquad \quad \  {\text{if}} \ \  s > 3/2 \\
1- \qquad   {\text{if}} \ \ s = 3/2 \\
\frac{s}{2}+\frac 14 \quad  \ \,   {\text{if}} \ \ 1/2 \leq s < 3/2 \\
s \qquad \quad \ {\text{if}} \ \ 0 < s < 1/2
\end{cases} \, .
$$
One verifies in a straightforward way that for any $s > 0$, 
\begin{equation}\label{identities for tau_2}
\min \{ s, \, s  - \tau_2(s) + \tau(s - \tau_2(s)) \}=s \, ,
\end{equation}
where $\tau(s)$ is defined in \eqref{tau}.
 We then obtain the following corollary of Theorem  \ref{Theorem Phi - Phi_L} and Theorem \ref{diffeogauge}.
 \begin{corollary}\label{approximation of derivative of Phi_0}
$(i)$ For any $s \ge 0$, $\Phi - \Phi_0 : H^s_{r,0} \to \frak h^{s+ \frac12 + \tau(s)}$ is real analytic.
As a consequence, for any $u \in H^s_{r,0}$, $s \ge 0$, $d_u \Phi_0$ is a high frequency approximation
of $d_u \Phi$, i.e., for any $u \in H^s_{r,0}$ with $s \ge 0$, 
$$
d_u \Phi - d_u \Phi_0 :  H^s_{r,0} \to \frak h^{s+ \frac12 + \tau(s)}
$$
is a bounded linear operator.

\noindent
$(ii)$ For any $u \in H^s_{r,0}$ with $s > 0$, $(d_u \Phi_0)^{-1}$ is  a high frequency approximation of $(d_u \Phi)^{-1}$ in the sense that
$(d_u \Phi)^{-1} - (d_u \Phi_0)^{-1}$
maps $ \frak h^{s+ \frac12 - \tau_2(s)}$ into $H^s_{r,0}$ and 
$$
(d_u \Phi)^{-1} - (d_u \Phi_0)^{-1} :  \frak h^{s + \frac12 - \tau_2(s)} \to H^s_{r,0}
$$ 
is bounded.
\end{corollary}
\begin{proof} $(i)$ By the above considerations, $\Phi - \Phi_0 : H^s_{r,0} \to \frak h^{s+ \frac12}$ is real analytic for any $s \ge 0$.
In particular, each component of $\Phi - \Phi_0$ is a real analytic map $H^s_{r,0} \to \C$. 
Since by Theorem \ref{Theorem Phi - Phi_L},   $\Phi - \Phi_0 : H^s_{r,0} \to \frak h^{s+ \frac12 + \tau(s)}$ is continuous for any $s \ge 0$, 
one infers from \cite[Theorem A.5]{KP1} that $\Phi - \Phi_0 : H^s_{r,0} \to \frak h^{s+ \frac12 + \tau(s)}$ is real analytic.

\noindent
$(ii)$ For any given $u \in H^s_{r,0}$ with $s > 0$, introduce the linear operators 
$$
A(u):= d_u \Phi - d_u \Phi_0 : L^2_{r,0} \to \frak h^{ \frac12 } \, ,  \quad
$$
$$
B(u):= d_u \Phi^{-1} - d_u \Phi_0^{-1}: \frak h^{\frac12} \to L^2_{r,0} \, . 
$$
Note that
$$
\begin{aligned}
\mbox{Id} & = d_u \Phi  \circ (d_u \Phi)^{-1} = d_u \Phi \circ ((d_u \Phi_0)^{-1} + B(u)) \\
& =  (d_u \Phi_0 + A(u)) \circ (d_u \Phi_0)^{-1} +   d_u \Phi \circ B(u) \\
& = \mbox{Id} + A(u) \circ (d_u \Phi_0)^{-1} +  d_u \Phi \circ B(u) \, .
\end{aligned}
$$
It then follows that
$$
B(u) = -  (d_u \Phi)^{-1}  \circ A(u) \circ (d_u \Phi_0)^{-1} \, .
$$
Furthermore, by item $(i)$ and \eqref{identities for tau_2}, $A(u)$ maps $H^{s- \tau_2(s)}_{r,0}$ into $\frak h^{s+ \frac12 }$
and
\begin{equation}\label{properties A(u)}
A(u) : H^{s- \tau_2(s)}_{r,0} \to \frak h^{s+ \frac12 } \, ,
\end{equation}
is bounded. Since 
$(d_u \Phi_0)^{-1} : \frak h^{s + \frac12 - \tau_2(s)} \to H^{s - \tau_2(s)}_{r,0}$ 
and $(d_u \Phi)^{-1} :  \frak h^{s + \frac12} \to H^{s}_{r,0}$ are bounded linear operators,
we then conclude that $B(u)$ maps  $\frak h^{s + \frac12 - \tau_2(s)}$ into $H^{s }_{r,0}$ and that
$B(u) :  \frak h^{s + \frac12 - \tau_2(s)} \to H^{s }_{r,0}$ is bounded.
\end{proof}


\end{document}